\documentclass[11pt,a4paper,twoside]{article}
\usepackage[naustrian,UKenglish]{babel}
\usepackage[T1]{fontenc}
\usepackage[utf8x]{inputenc}
\usepackage{latexsym,amsfonts,amsmath,amsthm,amssymb, mathrsfs}
\usepackage{fullpage}
\usepackage{xcolor,bm, bbm}
\usepackage{paralist}
\usepackage[shortcuts]{extdash}
\usepackage{todonotes}

\usepackage{fourier}

\newcommand*{\N}[0]{\mathbb{N}}
\newcommand*{\R}[0]{\mathbb{R}}
\newcommand*{\GenDist}[0]{P_{p,R^{(n)}}^{(n)}}
\newcommand*{\X}[0]{X_p^{(n)}}
\newcommand*{\Sph}[2]{\mathbb{S}_{#1}^{#2}}
\newcommand*{\Kug}[2]{\mathbb{B}_{#1}^{#2}}
\newcommand*{\Colon}[0]{\;\mathord{:}\;}
\newcommand*{\Ratio}[2]{\mathscr{R}_{#1}^{#2}}
\newcommand*{\ee}[0]{\mathrm{e}}
\newcommand*{\KiVert}[1]{\xrightarrow[#1]{\text{\upshape d}}}
\newcommand*{\dKol}[0]{d_{\text{\upshape Kol}}}
\newcommand*{\dd}[0]{\mathrm{d}}
\newcommand*{\TV}[0]{\text{\upshape TV}}
\newcommand*{\GlVert}[0]{\mathrel{\stackrel{\text{\upshape d}}{=}}}
\newcommand*{\Gaussp}[1]{\gamma_{#1}}

\renewcommand{\epsilon}{\varepsilon}
\renewcommand{\phi}{\varphi}

\DeclareMathOperator{\prb}{\mathbb{P}}
\DeclareMathOperator{\Erw}{\mathbb{E}}
\DeclareMathOperator{\Var}{\mathbb{V}}
\DeclareMathOperator{\cov}{Cov}
\DeclareMathOperator{\I}{\mathbb{I}}
\DeclareMathOperator{\U}{Unif}
\DeclareMathOperator{\Borel}{\mathcal{B}}
\DeclareMathOperator{\Leb}{\lambda}
\DeclareMathOperator{\Surf}{\sigma}
\DeclareMathOperator{\Keg}{\mu}
\DeclareMathOperator{\Normal}{\mathcal{N}}

\newtheorem{thm}{Theorem}[section]
\newtheorem{cor}[thm]{Corollary}
\newtheorem{df}[thm]{Definition}
\newtheorem{proposition}[thm]{Proposition}
\newtheorem{lem}[thm]{Lemma}

\newtheorem{thmalpha}{Theorem}

\newtheorem{AssumptionA}{Assumption}

\theoremstyle{definition}
\newtheorem{rem}[thm]{Remark}

\allowdisplaybreaks

\begin{document}

\title{\bf The Maclaurin inequality through the probabilistic lens}
\medskip

\author{Lorenz Fr\"uhwirth, Michael Juhos, and Joscha Prochno}
\date{}



\maketitle

\begin{abstract}
In this paper we take a probabilistic look at Maclaurin's inequality, which is a refinement of the classical AM--GM mean inequality. In a natural randomized setting, we obtain limit theorems and show that a reverse inequality holds with high probability. The form of Maclaurin's inequality naturally relates it to U-statistics. More precisely, given $x_1, \dotsc, x_n, p \in(0,\infty)$ and \(k \in \N\) with \(k \leq n\), let us define the quantity $$S_{k, p}^{(n)} = \Bigl( \tbinom{n}{k}^{-1} \sum_{1 \leq i_1 < \dotsb < i_k \leq n} x_{i_1}^p \dotsm x_{i_k}^p \Bigr)^{1/(k p)}.$$ Then as a consequence of the classical Maclaurin inequalities, we know that \(S_{k_1}^{(n)} \geq S_{k_2}^{(n)}\) for \(k_1 < k_2\). In the present article we consider the ratio
\[
\mathscr{R}_{k_1, k_2, p}^{(n)} := \frac{S_{k_2, p}^{(n)}}{S_{k_1, p}^{(n)}},
\]
evaluated at a random vector \((X_1, \dotsc, X_n)\) sampled either from the normalized surface measure on the \(\ell_p^n\)-sphere or from a distribution generalizing both the uniform distribution on the \(\ell_p^n\)-ball and the cone measure on the \(\ell_p^n\)-sphere; by the Maclaurin inequality, we always have \(\mathscr{R}_{k_1, k_2, p}^{(n)} \leq 1\). We derive central limit theorems for \(\mathscr{R}_{k_1, k_2, p}^{(n)}\) and \(\mathscr{R}_{k_1, n, p}^{(n)}\) as well as Berry--Esseen bounds and a moderate deviations principle for \(\mathscr{R}_{k_1, n, p}^{(n)}\), keeping \(k_1\), \(k_2\) fixed, in order to quantify the set of points where \(\mathscr{R}_{k_1, k_2, p}^{(n)} > c\) for \(c \in (0, 1)\), i.e., where the Maclaurin inequality is reversed up to a factor. The present aricle partly generalizes results concerning the AM--GM inequality obtained by Kabluchko, Prochno, and Vysotsky (2020), Th\"ale (2021), and Kaufmann and Th\"ale (2023+).
\end{abstract}

\section{Introduction and main results}

\subsection{Introduction}

There are a number of classical inequalities frequently used throughout mathematics. A natural question is then to characterize the equality cases or to determine to what degree a reverse inequality may hold. The latter shall be the main focus here, and we begin by motivating and illustrating the approach in the case of the classical inequality of arithmetic and geometric means (AM\==GM inequality), which has attracted quite some attention in the past decade. Let us recall that the AM\==GM inequality states that for any finite number of non-negative real values, the geometric mean is less than or equal to the arithmetic mean. More precisely, for all $n \in \N$ and $x_1, \dotsc, x_n \geq 0$ it holds that
\begin{equation*}
	\label{EqClassicalArithGeoIneq}
	\Bigl( \prod_{i=1}^n x_i \Bigr)^{1/n} \leq  \frac{1}{n} \sum_{i=1}^n x_i
\end{equation*}
and equality holds if and only if $x_1= \dotsb = x_n$. Setting $y_i := \sqrt{x_i}$ for $i=1, \dotsc,n$, we obtain 
\begin{equation*}
	\Bigl( \prod_{i=1}^n y_i \Bigr)^{1/n} \leq  \Bigl( \frac{1}{n} \sum_{i=1}^n y_i^2 \Bigr)^{1/2}.
\end{equation*}
For a point $y$ in the Euclidean unit sphere $\Sph{2}{n-1} := \{x=(x_i)_{i=1}^n\in\R^n \Colon \sum_{i=1}^n \lvert x_i\rvert^2 = 1\}$, this leads to the estimate
\begin{equation*}
	\Bigl( \prod_{i=1}^{n} \lvert y_i\rvert\Bigr)^{1/n} \leq \frac{1}{\sqrt{n}}.
\end{equation*}
It is natural to ask whether this inequality can be reversed for a ``typical'' point in $\Sph{2}{n-1}$, and Gluskin and Milman~\cite[Proposition~1]{GluMil} showed that, for any $t > 0$,
\begin{equation*}
	\label{EqGluMilResult}
	\Surf^{(n)}_2 \Bigl( \Bigl\{ x \in \Sph{2}{n-1} \Colon \Bigl( \prod_{i=1}^{n} \lvert x_i\rvert \Bigr)^{1/n}  \geq t \cdot \frac{1}{\sqrt{n}} \Bigr\} \Bigr) \geq 1- (1.6 \sqrt{t})^n,
\end{equation*}
where $\Surf^{(n)}_2$ denotes the unique rotationally invariant probability measure on $\Sph{2}{n-1}$.
For large dimensions $n \in \N$ this means that with high probability, we can reverse the AM\==GM inequality up to a constant. The problem was revisited later by Aldaz~\cite[Theorem 2.8]{Aldaz} who showed that for all $\epsilon > 0$ and $k> 0$ there exists an $N \in \N$ such that for every $n \geq N$,
\begin{equation*}
	\label{EqAldazResult}
	\Surf^{(n)}_2\Biggl( \Biggl\{  x \in \Sph{p}{n-1} \Colon \frac{(1- \epsilon) \ee^{ -\frac{1}{2} ( \gamma + \log 2)} }{\sqrt{n}}  <  \Bigl( \prod_{i=1}^{n} \lvert x_i\rvert \Bigr)^{1/n} < \frac{(1 + \epsilon) \ee^{ -\frac{1}{2} ( \gamma + \log 2)} }{\sqrt{n}} \Biggr\} \Biggr) \geq 1- \frac{1}{n^k},
\end{equation*}
where $\gamma=0.5772\ldots$ is Euler's constant. The previous works then motivated Kabluchko, Prochno, and Vysotsky~\cite{ArithGeoIneq} to study the asymptotic behavior of the $p$-generalized AM\==GM inequality, which states that for any $p \in (0, \infty)$, $n \in \N$, and $(x_i)_{i=1}^n \in \R^n$, 
\begin{equation*}
	\label{EqIntrpgenArithGeoIneq}
	\Bigl( \prod_{i=1}^{n} \lvert x_i\rvert \Bigr)^{1/n} \leq \Bigl( \frac{1}{n} \sum_{i=1}^{n} \lvert x_i\rvert^p \Bigr)^{1/p}.
\end{equation*} 
The authors analyzed the quantity
\begin{equation*}
	\label{EqRatioArithGeo}
	\Ratio{}{(n)} := \frac{\bigl( \prod_{i=1}^{n} \lvert x_i\rvert \bigr)^{1/n}  }{\lVert x\rVert_p},
\end{equation*}
where $ \lVert x\rVert_p := \bigl( \frac{1}{n} \sum_{i=1}^{n} \lvert x_i\rvert^p  \bigr)^{1/p} $ for $x=(x_i)_{i=1}^n\in\R^n$. Similar to the case of the classical AM\==GM inequality above it is natural to consider points $x \in \R^n$ that are in some sense uniformly distributed on the $\ell_p^n$ unit sphere $\Sph{p}{n-1}$ or the $\ell_p^n$ unit ball $\Kug{p}{n}$ respectively, where
\begin{equation*}
	\Kug{p}{n} := \bigl\{ x \in \R^n \Colon \lVert x \rVert_p \leq 1 \bigr\} \quad \text{and} \quad \Sph{p}{n-1} := \bigl\{ x \in \R^n \Colon \lVert x \rVert_p = 1 \bigr\}.
\end{equation*} 
In \cite[Theorem 1.1]{ArithGeoIneq} it is shown that, for a constant $m_p\in(0,\infty)$ depending only on $p$,
\begin{equation*}
	\sqrt{n} \bigl( \ee^{- m_p} \Ratio{}{(n)} -1   \bigr), \quad n \in \N,
\end{equation*}
converges to a centered normal distribution with known variance, and in \cite[Theorem 1.3]{ArithGeoIneq} a large deviations principle (LDP) for the sequence $(\Ratio{}{(n)})_{n \in \N }$ is proved (see Section~\ref{SecLDPProb} for the definition of an LDP).

The work \cite{ArithGeoIneq} of Kabluchko, Prochno, and Vysotsky was then complemented by Th\"ale~\cite{Th2021}, who obtained a Berry\--Esseen type bound and a moderate deviations principle for a wider class of distributions on the $\ell_p^n$ balls (first defined in~\cite{Bartheetal}). In a subsequent paper, Kaufmann and Th\"ale~\cite[Theorem 1.1]{KaufThaele} were able to identify the sharp asymptotics of $(\Ratio{}{(n)})_{n \in \N }$.

Quite recently, Fr\"uhwirth and Prochno~\cite{HoelderIneq} approached another fundamental inequality in analysis, namely H\"older's inequality, by probabilistic means. More precisely, they studied the asymptotic behavior of the classical H\"older ratio $\Ratio{p,q}{(n)}$ defined by
\begin{equation*}
\Ratio{p,q}{(n)} := \frac{ \sum_{i=1}^{n} \lvert x_i y_i\rvert}{\lVert x\rVert_p \lVert y\rVert_q},\quad n\in\N,
\end{equation*}
where $x,y \in \R^n$ and $p,q \in (1, \infty) $ with $ \frac{1}{p}+ \frac{1}{q}=1$. The points $x,y \in \R^n$ were drawn at random according to the uniform distribution on the $\ell_p^n$ and $\ell_q^n$ unit ball, respectively. In~\cite{HoelderIneq}, a central limit theorem and large and moderate deviations principles are proved for the sequence $(\Ratio{p,q}{(n)})_{n \in \N}$.

The aim of this article is to study a refined version of the AM\==GM inequality, known as Maclaurin's inequality. It is another classical inequality in analysis and states that if for a point $x = (x_i)_{i = 1}^n \in \R^n$,  \(k \in \{1, 2, \dotsc, n\}\), and $p \in (0, \infty)$ one defines the \emph{\(p\)\-/generalized \(k\)\textsuperscript{th} symmetric mean}
\begin{equation*}
S_{k, p}^{(n)} := S_{k, p}^{(n)}(x) := \left ( \binom{n}{k}^{-1} \sum_{1 \leq i_1 < \dotsb < i_{k} \leq n}  \lvert x_{i_1}\rvert^p \dotsm \lvert x_{i_{k}} \rvert^p \right)^{1/(k p)},
\end{equation*}
then one has that
\begin{equation}
\label{EqEulerMclIneq}
 S_{k,p}^{(n)} \geq S_{k + 1, p}^{(n)},
\end{equation}
where $1 \leq k \leq n-1$. By iterating Equation~\eqref{EqEulerMclIneq} over all $1 \leq k \leq n-1$, one recovers the $p$-generalized AM\==GM inequality, i.e., we have 
\begin{equation*}
\left( \prod_{i=1}^{n} \lvert x_i \rvert \right)^{1/n} = S_{n,p}^{(n)} \leq S_{n-1,p}^{(n)} \leq \dotsb \leq S_{1,p}^{(n)} = \left( \sum_{i=1}^{n} \lvert x_i\rvert^p \right)^{1/p}.
\end{equation*}
This motivates the study of the Maclaurin ratio 
\begin{equation}
\label{EqDefRatioEulerMclaurin}
\Ratio{k_1, k_2, p}{(n)} := \Ratio{k_1, k_2, p}{(n)}(\X) := \frac{S_{k_2,p}^{(n)}}{S_{k_1,p}^{(n)}},
\end{equation}
with $1 \leq k_1 < k_2 \leq n$ and $\X \in \R^n$ a random point; note \(\Ratio{k_1, k_2, p}{(n)} \in [0, 1]\) because of~\eqref{EqEulerMclIneq}. As already explained before, it is natural to consider $\X \sim \U(\Kug{p}{n})$ or $\X \sim \U(\Sph{p}{n-1}) $. The uniform distribution on $\Kug{p}{n}$ is given by the normalized Lebesgue measure, whereas there are two meaningful uniform distributions on $\Sph{p}{n-1}$, namely the surface probability measure and the cone probability measure. In the present article, we are going to treat \(\U(\Kug{p}{n})\) and the cone measure on \(\Sph{p}{n - 1}\) in a unified way by introducing a family of distributions denoted by \(\GenDist\): this is defined to be the distribution of the random vector \(R^{(n)} Y\), where \(R^{(n)}\) and \(Y\) are independent, \(R^{(n)}\) almost surely takes values in \((0, \infty)\), and \(Y\) is distributed according to the cone measure on \(\Sph{p}{n - 1}\); clearly \(R^{(n)}\) is the radial part of the distribution, while \(Y\) is its directional part.

\subsection{Main results\---limit theorems for the Maclaurin ratio}
\label{SubSectionMainResults}

Let us now present our main results. For the sake of brevity, we first introduce the following general assumption on our random quantity.

\begin{AssumptionA}
\label{AssA}
Let $p \in [1, \infty) $ and $\X \in \R^n$ be a random vector so that either $\X \sim \GenDist$ or $\X \sim \Surf_p^{(n)}$, where $\Surf_p^{(n)}$ is the surface probability measure on $\Sph{p}{n-1}$.
\end{AssumptionA}

In the statements of the main results, the following quantities will appear. For a real\-/valued random variable $X$ distributed according to the $p$\=/Gaussian distribution (the definition of which is given in Section~\ref{sec:meas_lp_spaces}), we set
\begin{equation}\label{DefRelQuant}
\begin{split}
m_p &:= \Erw[\log\lvert X\rvert] \in \Bigl[ -1, -\frac{1}{2} \Bigr],\\
s_p^2 &:= \Var\Bigl[ \log\lvert X\rvert - \frac{1}{p} \lvert X\rvert^p \Bigr] \in (0, \infty),\\
\rho_p^2 &:= \frac{1}{4} \Var\bigl[ \bigl( \lvert X \rvert^p - 1 - p \bigr)^2 \bigr] \in (0, \infty).
\end{split}
\end{equation}

\subsubsection{Central limit theorems and Berry\--Esseen bounds for $\Ratio{k_1, k_2, p}{(n)}$}

\begin{thmalpha}
\label{ThmClt1}
Let $k \in \N$ be fixed, let $p \in [1, \infty) $, and for each \(n \geq k\), we let $ \X $ be a random vector in $\R^n$ so that Assumption~\ref{AssA} holds. Then, the Maclaurin ratio $\Ratio{k, n, p}{(n)}$ satisfies 
\begin{equation*}
\sqrt{n} \left( \ee^{- m_p} \Ratio{k, n, p}{(n)} -1 \right) \KiVert{n \to \infty} Z,
\end{equation*}
where $ Z \sim \Normal(0, s_p^2) $.
\end{thmalpha}

We are also able to provide a quantitative version of Theorem~\ref{ThmClt1}, i.e., a Berry\--Esseen type result. For real-valued random variables \(X\) and \(Y\) on a common probability space, their Kolmogorov distance is defined by
\begin{equation}
    \label{EqKolDist}
    \dKol(X,Y) := \sup_{t \in \R} \bigl\lvert \prb[X \leq t] - \prb[Y \leq t] \bigr\rvert.
\end{equation}

\begin{thmalpha}
\label{ThmBE1}
Let $k \in \N$ be fixed, let $p \in [1, \infty) $, and for each \(n \geq k\), let $ \X $ be a random vector in $\R^n$ so that Assumption~\ref{AssA} holds. Then there exists a constant $c_{k, p} \in (0, \infty)$ such that, for all $n \geq k$,
\begin{equation*}
\dKol\left( \sqrt{n} \left( \ee^{- m_p} \Ratio{k, n, p}{(n)} -1 \right), Z \right) \leq \frac{c_{k, p} \sqrt{\log(n)}}{\sqrt{n}},
\end{equation*}
where $Z \sim \Normal(0, s_p^2)$. 
\end{thmalpha}

The proof of Theorem~\ref{ThmBE1} uses Taylor expansion in order to write \(\Ratio{k, n, p}{(n)}\) as a sum of a certain U\=/statistic, for which a Berry\--Esseen result is available already, and the remainder term which is treated as a perturbation whose behaviour is controlled with the help of Theorem~\ref{ThmMDPRkn} given below. 

\begin{rem}
Theorem~\ref{ThmBE1} gives a similar asymptotic bound of the distance to a normal distribution as the classical theorem of Berry\--Esseen. The ``$\sqrt{\log(n)}$'' on the right\-/hand side seems to be an artifact of our method of proof and we conjecture that this factor is not necessary.
\end{rem}

An almost immediate consequence of Theorem~\ref{ThmClt1} is the following answer to the question stated in the introduction in the context of the AM\==GM\-/inequality, in how much the Maclaurin inequality can be reversed.

\begin{cor}\label{cor:CLT1}
Under the conditions of Theorem~\ref{ThmClt1} we have
\begin{equation*}
\lim_{n \to \infty} \prb\bigl[ S_{n, p}^{(n)} \geq \ee^{m_p} S_{k, p}^{(n)} \bigr] = \frac{1}{2},
\end{equation*}
that is, with probability approaching \(\frac{1}{2}\), we can reverse Maclaurin's inequality up to the factor \(\ee^{m_p}\).
\end{cor}

Whereas Theorem~\ref{ThmClt1} compares the \(k\)\textsuperscript{th} symmetric mean against the geometric mean (for fixed \(k\)), the next theorem compares the \(k_1\)\textsuperscript{th} against the \(k_2\)\textsuperscript{th} symmetric mean (for fixed \(k_1 \neq k_2\)) and provides a CLT.

\begin{thmalpha}
\label{ThmClt2}
Let $k_1, k_2 \in \N$ be fixed with $k_1 < k_2$, let $p \in [1, \infty) $, and for each \(n \geq k_2\) let $ \X \sim \Keg_p^{(n)}$ be a random vector in $\R^n$. Then, the Maclaurin ratio $ \Ratio{k_1, k_2, p}{(n)}$ satisfies
\begin{equation*}
\sqrt{n} \left( \frac{np}{k_2 - k_1} \left(\Ratio{k_1, k_2, p}{(n)} - 1 \right) + \frac{p}{2} \right) \KiVert{n \to \infty} Z,
\end{equation*}
where $Z \sim \Normal(0, \rho_p^2)$.
\end{thmalpha}

Theorem~\ref{ThmClt2} basically is a consequence of a variation of the delta\-/method and the CLT for U\=/statistics. (Theorem~\ref{ThmClt1} too can be proved directly in this manner instead of via Theorem~\ref{ThmBE1}.) Concerning reversal of Maclaurin's inequality we can make the following statement.

\begin{cor}\label{cor:CLT2}
Under the premises of Theorem~\ref{ThmClt2}, for any \(c \in (0, 1)\), we know
\begin{equation*}
\lim_{n \to \infty} \prb\bigl[ S_{k_2, p}^{(n)} \geq c S_{k_1, p}^{(n)} \bigr] = 1.
\end{equation*}
This means, with probability approaching \(1\), Maclaurin's inequality can be reversed up to a constant factor \(c\) arbitrarily close to \(1\).
\end{cor}

\subsubsection{Moderate deviations principles for $\Ratio{k, n, p}{(n)}$}

\begin{thmalpha}[Moderate deviations]
\label{ThmMDPRkn}
Let $k \in \N$ be fixed, $p \in [1, \infty) $, and for each \(n \geq k\) let $\X$ be a random vector in $\R^n $ such that Assumption~\ref{AssA} holds. Further, let $(b_n)_{n \in \N} $ be a sequence of positive real numbers such that \(\lim_{n \to \infty} \frac{b_n}{\sqrt{n}} = 0\), and either \(\lim_{n \rightarrow \infty } b_n = \infty\) if \(\X \sim \Keg_p^{(n)}\) for all \(n \in \N\), or \(\lim_{n \to \infty} \frac{b_n}{\sqrt{\log(n)}} = \infty\) if \(\X \sim \Surf_p^{(n)}\) for all \(n \in \N\). Then $\Bigl( \frac{\sqrt{n}}{b_n} \bigl( \ee^{-m_p} \Ratio{k,n,p}{(n)} - 1 \bigr) \Bigr)_{n \geq k} $ satisfies a moderate deviations principle at speed $(b_n^2)_{n \in \N}$ with good rate function $\I \colon \R \rightarrow [0, \infty] $ given by $\I(t) := \frac{t^2}{2 s_p^2}$.
\end{thmalpha}

The main point in the proof of Theorem~\ref{ThmMDPRkn} is to establish exponential equivalence of the expressions \(\binom{n}{k}^{-1} \sum_{i_1 < \dotsb < i_k} X_{i_1} \dotsm X_{i_k}\) and \(\bigl( \frac{1}{n} \sum_{i = 1}^k X_i \bigr)_k\), thereby reducing a problem with dependent summands to one with independent summands, where moderate deviations are well\-/understood. 

\begin{rem}
We emphasize the fact that Theorems~\ref{ThmClt1}, \ref{ThmBE1}, and \ref{ThmMDPRkn} are ``oblivious'' to the parameter \(k\), i.e., the limiting objects to not depend on (fixed!)\ \(k\) and take the same forms as for the known case \(k = 1\), i.e.,\ the AM\--GM setting; even more, shift and scaling of the quantity under consideration, \(\ee^{-m_p} \Ratio{k, n, p}{(n)} - 1\), make no use of \(k\). Compare this to Theorem~\ref{ThmClt2}, where the limiting distribution indeed does not depend on \(k_1\) or \(k_2\), but at least the difference \(k_2 - k_1\) enters the proper scaling.
\end{rem}

\begin{rem}
The question of large deviations principles for \(\Ratio{k, n, p}{(n)}\) remains open still, as does that of further limit theorems for \(\Ratio{k_1, k_2, p}{(n)}\). The reason is a lack of suitable tools; there do exist general results concerning LDPs for U\=/statistics, but their strict conditions are not met by our setting. For more details we refer the reader to Remark~\ref{rem:mdp_ldp_ustat}.
\end{rem}

\subsection*{Organization of the paper}

The remainder of this article is structured as follows. In Section~\ref{sec:notation and prelim} we lay out the mathematical background needed throughout this article; this entails basics from probability theory including moderate and large deviations; measures related to finite\-/dimensional \(\ell_p\)\=/spaces; and the theory of U\=/statistics. In Section~\ref{SectionProofsPrelim} follow the proofs of the main results and corollaries as stated above, together with a few lemmas.

\section{Notation and preliminaries}
\label{sec:notation and prelim}

We shall now briefly introduce the notation used throughout the text together with some background material on large deviations and some further results used in the proofs.

\subsection{Notation}

We set \(\R_+ := (0, \infty)\). The standard inner product on \(\R^n\) is written \(\langle \cdot, \cdot \rangle\), and the Euclidean norm \(\lVert \cdot \rVert_2\). The induced Borel \(\sigma\)\=/algebra is denoted by \(\Borel(\R^n)\), and measurability by default refers thereto.

We shall suppose that all random variables are defined on a common probability space whose probability measure is written \(\prb\); the expectation, variance, and covariance w.r.t.\ \(\prb\) are denoted by \(\Erw\), \(\Var\), and \(\cov\), respectively.

`Almost surely' is abbreviated `a.s.' and is to be understood w.r.t.\ \(\prb\); the rare occurrence of `a.e.' means `almost everywhere,' the reference measure always being explicit; `i.i.d.'\ means `independent and identically distributed.'

Convergence in distribution is indicated by \(\KiVert{}\); \(\GlVert\) means equality in distribution. The formula \(X \sim \mu\) means, the random variable \(X\) has probability distribution \(\mu\). The normal (Gaussian) distribution with mean \(m\) and variance (or covariance matrix) \(s^2\) is written \(\Normal(m, s^2)\).

\subsection{Basics from probability and large deviations theory}
\label{SecLDPProb}

Let $d \in \N $ and $( \xi_n)_{n \in \N}$ be a sequence of $\R^d$\=/valued random variables and let $(s_n)_{n \in \N}$ be a sequence of real numbers tending to infinity. A function \(\I \colon \R^d \to [0, \infty]\) is called a good rate function (GRF) iff \(\I\) is not constant \(\infty\) and the sublevel set \(\I^{-1}([0, a])\) is compact for any \(a \in [0, \infty)\). We say that $( \xi_n)_{ n \in \N}$ satisfies a large deviations principle (LDP) in $\R^d$ at speed $(s_n)_{n \in \N }$ with GRF \(\I\) if and only if
\begin{equation}
	\label{EqIntrLDP}
	-\inf_{ x \in A^{ \circ }} \I( x) \leq \liminf_{ n \rightarrow \infty } \frac{1}{s_n} \log \prb[ \xi_n \in A] \leq \limsup_{ n \rightarrow \infty } \frac{1}{s_n} \log \prb[ \xi_n \in A] \leq - \inf_{ x \in  \overline{A}} \I(x)
\end{equation}
for all $A\in \Borel(\R^d)$, where \(A^\circ\) and \(\overline{A}\) denote the interior and the closure of \(A\), respectively.

There are different ways to show that a certain sequence of random variables satisfies an LDP, one of the most commonly used is the so called contraction principle (see, e.g., Theorem 4.2.1 in \cite{DZ2011}).

\begin{proposition}[Contraction principle]
	\label{LemContractionPrinciple}
	Let $d,m\in\N$, let $f\colon \R^d \rightarrow \R^m$ be a continuous function, and let $( \xi_n)_{ n \in \N}$ be a sequence of \(\R^d\)\=/valued random variables that satisfies an LDP at speed $(s_n)_{n \in \N }$ with GRF $\I \colon \R^d \rightarrow [0, \infty ]$. Then, the sequence $( f( \xi_n))_{n \in \N}$ satisfies an LDP in $\R^m$ at speed $(s_n)_{n \in \N}$ with the GRF $\mathbb{J} \colon \R^m \rightarrow [0, \infty]$ given by
	\begin{equation*}
		\mathbb{J}(y) := \inf_{ x \in f^{-1}( \{ y \})} \I(x).
	\end{equation*}
\end{proposition}

We shall also use Cram\'er's large deviations theorem for $\R^d$\=/valued random variables (see, e.g., \cite[Corollary 6.1.6]{DZ2011}).

\begin{proposition}[Cram\'er's theorem]
	\label{ThmCramer}
	Let $(X_i)_{i \in \N}$ be a sequence of i.i.d.\ $\R^d$\=/valued random variables, write \(\Lambda(t) := \log \Erw[\ee^{\langle t, X_1 \rangle}]\) for \(t \in \R^d\), and suppose that $0 \in \mathcal{D}_{ \Lambda}^{ \circ }$, where
	\begin{equation*}
		\mathcal{D}_{ \Lambda} := \bigl\{ t \in \R^d \Colon \Lambda(t) < \infty \bigr\}.
	\end{equation*}
	Then, the sequence $(\xi_n)_{n \in \N}$ defined by
	\begin{equation*}
		\xi_n := \frac{1}{n} \sum_{i = 1}^n X_i, \quad n \in \N,
	\end{equation*}
	satisfies an LDP in $\R^d$ at speed $n$ with GRF $\Lambda^* \colon \R^d \rightarrow [0, \infty]$, where
	\begin{equation*}
		\Lambda^*(x) := \sup_{ t \in \R^d} \bigl[ \langle x, t \rangle - \Lambda(t) \bigr]
	\end{equation*}
	is the \emph{Fenchel\--Legendre transform} of \(\Lambda\).
\end{proposition}

Two sequences of $ \R^d$\=/valued random variables $(\xi_n)_{n \in \N}$ and $(\eta_n)_{n \in \N}$ are said to be exponentially equivalent at speed $(s_n)_{n \in \N}$ if and only if we have, for all $\epsilon > 0$,
\begin{equation}
	\label{EqExpEquivalence}
	\lim_{n \rightarrow \infty} \frac{1}{s_n} \log \prb\bigl[ \bigl\lVert \xi_n - \eta_n \bigr\rVert_2 > \epsilon \bigr] = -\infty.
\end{equation} 

The following result can be found, e.g., in \cite[Theorem 4.2.13]{DZ2011}; it states that if a sequence of random vectors satisfies an LDP and is exponentially equivalent to another sequence of random vectors, then both satisfy the same LDP.

\begin{proposition}
	\label{PropExpEquLDP}
	Let $(\xi_n)_{n \in \N}$ and $(\eta_n)_{n \in \N}$ be sequences of $\R^d$\=/valued random variables. Assume that $(\xi_n)_{n \in \N}$ satisfies an LDP at speed $(s_n)_{n \in \N}$ with GRF\ $\I \colon \R^d \rightarrow [0, \infty]$ and that $(\xi_n)_{n \in \N}$ and $(\eta_n)_{n \in \N}$ are exponentially equivalent at speed $(s_n)_{n \in \N}$. Then, $(\eta_n)_{n \in \N}$ satisfies an LDP at speed $(s_n)_{n \in \N}$ with the same GRF\ $\I \colon \R^d \rightarrow [0, \infty]$.
\end{proposition}

We say that a sequence $(\xi_n)_{n \in \N}$ of $\R^d$\=/valued random variables satisfies a moderate deviations principle (MDP) if and only if $\bigl( \frac{\xi_n}{\sqrt{n} b_n} \bigr)_{n \in \N} $ satisfies an LDP at speed $(b_n^2)_{n \in \N}$ with some GRF $\I \colon \R^d \rightarrow [0, \infty]$ for some positive sequence $(b_n)_{n\in\N}$ satisfying $\lim_{n \rightarrow \infty} b_n = \infty$ and $\lim_{n \rightarrow \infty} \frac{b_n}{\sqrt{n}}=0$.
The scaling $\sqrt{n} b_n$ typically grows faster than that in a central limit theorem but slower than that in a law of large numbers (and hence in an ``ordinary'' LDP). This property of an MDP is nicely illustrated by the following Cram\'er\-/type theorem (see, e.g., \cite[Theorem 3.7.1]{DZ2011}).

\begin{proposition}
	\label{ThmCramerMDP}
	Let $d \in \N$ and let $(X_i)_{i \in \N }$ be a sequence of i.i.d.\ $\R^d$\=/valued random variables such that \(0 \in \mathcal{D}_\Lambda^\circ\), as in Proposition~\ref{ThmCramer}, and suppose $\Erw[X_1]=0$ and that $\Sigma$, the covariance matrix of $X_1$, is invertible. Let $(b_n)_{n \in \N}$ be a sequence of real numbers with
	\begin{equation*}
		\lim_{n \rightarrow \infty} b_n = \infty \quad \text{ and } \quad \lim_{n \rightarrow \infty} \frac{b_n}{ \sqrt{n}} = 0.
	\end{equation*}
	Then, the sequence $(\xi_n)_{n \in \N}$ with $\xi_n := \frac{1}{b_n \sqrt{n}} \sum_{i=1}^n X_i$ satisfies an LDP in $\R^d$ at speed $( b_n^2)_{n \in \N}$ with GRF $\I \colon \R^d \rightarrow [0, \infty]$ given by
	\begin{equation*}
		\I(x) := \frac{1}{2} \langle x, \Sigma^{-1 } x \rangle, \quad x \in \R^d.
	\end{equation*}
\end{proposition}

The following statement, going back to Nagaev~\cite{Nagaev1969:1} and generalized by Gantert, Ramanan, and Rembart~\cite{GantertRamananRembart2014}, gives an asymptotic bound for sums of i.i.d.\ random variables with \emph{stretched exponential} tail probabilities; its formulation uses the following notion: a function $\ell \colon \R \rightarrow \R_+$ is called slowly varying (at infinity) if we have, for all $a \in \R_+$,
\begin{equation*}
\lim_{x \rightarrow \infty} \frac{ \ell(a x)}{ \ell (x) } = 1.
\end{equation*}

\begin{proposition}
\label{LemStrechedExp}
Let $(X_i)_{i \in \N}$ be a sequence of i.i.d.\ positive random variables and assume that there exist slowly varying functions $c, b \colon \R \rightarrow \R_+ $ and constants $r \in (0,1)$ and \(t^* \in \R_+\) such that we have, for all $t \in [t^*, \infty)$,
\begin{equation*}
\prb\left[ X_1 \geq t \right] \leq c(t) \exp \left( - b(t) t^r \right).
\end{equation*}
Then, for $m := \Erw[X_1] $ and any $x \in (m, \infty)$, we get
\begin{equation*}
\limsup_{n \rightarrow \infty} \frac{1}{b(n) n^r} \log \left( \prb\left[ \frac{1}{n} \sum_{i=1}^n X_i \geq x \right] \right) \leq - (x -m)^r.
\end{equation*}
The analogous statement holds true for the lower tail \(\prb[X_1 \leq t]\) and \(\prb\bigl[ \frac{1}{n} \sum_{i = 1}^n X_i \leq x \bigr]\) with \(x < m\).
\end{proposition}

The following result is taken from \cite[Lemma 4.1]{APTGaussianFluct}; it allows to estimate the Kolmogorov distance of a perturbed random variable to a Gaussian random variable.

\begin{proposition}
\label{PropEstimateKolmogoroff}
Let $Y_1$, $Y_2$ be random variables, let $Z \sim \Normal(0, \sigma^2)$ with $\sigma^2 \in \R_+$, and let $\epsilon \in \R_+$. Then,
\begin{equation*}
	\sup_{t \in \R} \bigl\lvert \prb[Y_1 + Y_2 \geq t] - \prb[Z \geq t] \bigr\rvert \leq \sup_{t \in \R} \bigl\lvert \prb[Y_1 \geq t] - \prb[Z \geq t] \bigr\rvert + \prb[\lvert Y_2 \rvert > \epsilon] + \frac{\epsilon}{\sqrt{2 \pi \sigma^2}}.
\end{equation*}
\end{proposition}

\subsection{Measures on finite\-/dimensional \(\ell_p\) spaces}
\label{sec:meas_lp_spaces}

For \(p \in [1, \infty)\) (the case \(p \in (0, 1) \cup \{\infty\}\) is not relevant for the present article) and \(n \in \N\) the space \(\ell_p^n\) is \(\R^n\) endowed with the norm \(\lVert \cdot \rVert_p\) defined by
\begin{equation*}
\lVert (x_1, \dotsc, x_n) \rVert_p := \Biggl( \sum_{i = 1}^n \lvert x_i \rvert^p \Biggr)^{1/p};
\end{equation*}
the corresponding unit ball and unit sphere are denoted by
\begin{equation*}
\Kug{p}{n} := \{x \in \R^n \Colon \lVert x \rVert_p \leq 1\} \qquad \text{and} \qquad \Sph{p}{n - 1} := \{x \in \R^n \Colon \lVert x \rVert_p = 1\}.
\end{equation*}
The standard measure on \(\R^n\) is the \(n\)\=/dimensional Lebesgue measure \(\Leb^{(n)}\), defined on \(\Borel(\R^n)\); it induces the uniform distribution on \(\Kug{p}{n}\).

We equip $\Sph{p}{n-1}$ with the trace \(\sigma\)\=/algebra induced by \(\Borel(\R^n)\), denoted \(\Borel(\Sph{p}{n - 1})\). There are two meaningful ``uniform'' distributions on $\Sph{p}{n-1}$, namely the cone probability measure $\Keg^{(n)}_p$ and the surface probability measure $\Surf^{(n)}_p$. In the following, we will briefly discuss their theoretical foundation as well as the relation between these two distributions. The cone probability measure $\Keg^{(n)}_p$ of $A \in \Borel(\Sph{p}{n-1})$ is defined as
\begin{equation*}
	\Keg^{(n)}_p(A):= \frac{\Leb^{(n)}(\{ r a \Colon r \in [0,1], a \in A\})}{\Leb^{(n)}(\Kug{p}{n})}.
\end{equation*}  
From a result of Schechtman and Zinn~\cite{SchechtZinn}, and Rachev and R\"uschendorf~\cite{RachevRuesch}, we know that $\X \sim \Keg^{(n)}_p$ if and only if
\begin{equation}
	\label{EqProbRepSchechtmannZinn}
	\X \GlVert \frac{Y^{(n)}}{\lVert Y^{(n)}\rVert_p},
\end{equation} 
where $Y^{(n)} = (Y_i)_{i = 1}^n$ is a vector of i.i.d.\ random variables whose common distribution is $\Gaussp{p}$, the \(p\)\=/generalized Gaussian distribution, or \(p\)\=/Gaussian for short; it is defined by its Lebesgue-density,
\begin{equation}
	\label{EqDensitypGenGaussian}
	\frac{\dd \Gaussp{p}}{\dd x}(x) = \frac{1}{2 p^{1/p} \Gamma( 1 + \frac{1}{p})} \ee^{ - \lvert x\rvert^p/p}, \quad x \in \R.
\end{equation} 

The surface probability measure $\Surf^{(n)}_p$ is defined to be the $(n-1)$-dimensional Hausdorff probability measure or, equivalently, the $(n-1)$-dimensional normalized Riemannian volume measure on $\Sph{p}{n-1}$. We have the following relation between $\Keg^{(n)}_p$ and $\Surf^{(n)}_p$ (see \cite[Lemma 2]{NaorRomikProjSurfMeasure} and \cite[Lemma 2.2]{ArithGeoIneq}).

\begin{proposition}\label{PropSurfKegDichte}
Let $n \in \N$ and $p \in [1, \infty)$. Then, for all $x = (x_1, \dotsc, x_n) \in \Sph{p}{n-1}$, 
\begin{equation*}
	\frac{\dd\!\Surf^{(n)}_p}{\dd\!\Keg^{(n)}_p}(x) = C_{n,p} \Biggl( \sum_{i = 1}^n \lvert x_i \rvert^{2 p - 2} \Biggr)^{1/2},
\end{equation*}
where
\begin{equation*}
	C_{n,p}:= \Biggl( \int_{\Sph{p}{n-1}} \Biggl( \sum_{i = 1}^n \lvert \xi_i \rvert^{2 p - 2} \Biggr)^{1/2} \, \dd\!\Keg^{(n)}_p(\xi) \Biggr)^{-1}.
\end{equation*}
It then follows, for all \(x \in \Sph{p}{n - 1}\),
\begin{equation*}
	n^{-\lvert 1/p - 1/2 \rvert} \leq \frac{\dd\!\Surf_p^{(n)}}{\dd\!\Keg_p^{(n)}}(x) \leq n^{\lvert 1/p - 1/2 \rvert}.
\end{equation*}
\end{proposition}

If $p \in \{1, 2\}$, there easily follows $\Surf^{(n)}_p=\Keg^{(n)}_p$. We remark that also in case $p = \infty $ we know $\Surf^{(n)}_{\infty}=\Keg^{(n)}_{\infty}$ (see \cite{NaorRomikProjSurfMeasure}). In contrast, for all $p \in (1, \infty)$ with $p \neq 2$, we have $\Surf^{(n)}_p \neq \Keg^{(n)}_p$. Nevertheless, for large $n \in \N$, one can prove that $\Surf^{(n)}_p$ and $\Keg^{(n)}_p$ are close in the total variation distance (see \cite[Theorem 2]{NaorRomikProjSurfMeasure}).

\begin{proposition}
\label{PropSurfMeasConeMeasClose}
For all $p \in [1, \infty)$ and all \(n \in \N\), we have
\begin{equation*}
	\lVert \Keg^{(n)}_p - \Surf^{(n)}_p \rVert_{\TV} := \sup_{ A \in \Borel(\Sph{p}{n-1}) } \lvert \Keg^{(n)}_p(A) - \Surf^{(n)}_p(A) \rvert \leq \frac{c_p}{\sqrt{n}},
\end{equation*} 
where $c_p \in \R_+$ depends only on $p$.
\end{proposition}

Barthe, Gu\'edon, Mendelson, and Naor~\cite{Bartheetal} have introduced the following unifying generalization of both \(\U(\Kug{p}{n})\) and \(\Keg_p^{(n)}\): take a sequence \((Y_i)_{i \in \N}\) of i.i.d.\ \(p\)\=/Gaussian random variables and a random variable \(W\) taking values in \([0, \infty)\), independent of \((Y_i)_{i \in \N}\); let \(n \in \N\). Then the random variable
\begin{equation}\label{EqGenDistBGMN}
X := \frac{(Y_i)_{i = 1}^n}{\bigl( \lVert (Y_i)_{i = 1}^n \rVert_p^p + W \bigr)^{1/p}}
\end{equation}
takes values in \(\Kug{p}{n}\). If \(W = 0\) almost surely, then \(X \sim \Keg_p^{(n)}\) as is immediately seen by comparison with Equation~\eqref{EqProbRepSchechtmannZinn}. If \(W\) follows an exponential distribution with rate \(\frac{1}{p}\), then \(X \sim \U(\Kug{p}{n})\). A complete characterization of all distributions on \(\Kug{p}{n}\), which have the probabilistic representation~\eqref{EqGenDistBGMN}, is hitherto unknown. Since it had been devised, authors have been using this generalized \(p\)\=/radial distribution in several works, for instance Kabluchko, Prochno, and Th\"ale~\cite{KPT2019_II} or Kaufmann, Sambale, and Th\"ale~\cite{KauSamThae2022}.

The ansatz \eqref{EqGenDistBGMN} has inspired further generalization by Kaufmann and Th\"ale~\cite{KaufThaele2022} who have introduced a homogeneous factor to the joint density of \((Y_i)_{i = 1}^n\) in order to also take into account eigenvalues or singular values of random matrices in Schatten classes. A different route has been taken by Heiny, Johnston and Prochno~\cite{HeiJoPro2022}, who have considered random vectors of the form \(X = R \Theta\), where \(R\) and \(\Theta\) are independent, \(R \in [0, \infty)\) almost surely, and \(\Theta \sim \Surf_2^{(n)}\); note that the distribution of \(X\) is invariant under orthogonal transformations.

In the present article we take up the latter thread for further generalization: Let \((Y_i)_{i \in \N}\) be i.i.d.\ \(p\)\=/Gaussian as before, let \(n \in \N\), and let \(R^{(n)}\) be an a.s.\ nonnegative real\-/valued random variable, independent of \((Y_i)_{i \in \N}\). Then we denote by \(\GenDist\) the distribution of the random vector
\begin{equation}\label{EqGenDistFJP}
\X := R^{(n)} \, \frac{(Y_i)_{i = 1}^n}{\lVert (Y_i)_{i = 1}^n \rVert_p},
\end{equation}
or, since \(\Theta_p^{(n)} := \frac{(Y_i)_{i = 1}^n}{\lVert (Y_i)_{i = 1}^n \rVert_p} \sim \Keg_p^{(n)}\) by Schechtman and Zinn, equivalently
\begin{equation*}
\X = R^{(n)} \Theta_p^{(n)}.
\end{equation*}
Obviously \(\lVert \X \rVert_p = R^{(n)}\) a.s.; if \(R^{(n)} = 1\) a.s., then \(\X \sim \Keg_p^{(n)}\), and if \(R^{(n)} \leq 1\) a.s., then \(\GenDist\) is a probability distribution on \(\Kug{p}{n}\). The following proposition collects some basic facts about \(\GenDist\); the proof is given in Section~\ref{ProofGenDistFJP}.

\begin{proposition}\label{PropGenDistFJP}
Let \(n \in \N\) and \(p \in [1, \infty)\).
\begin{compactenum}
\item Let \(\mu\) be a probability measure on \(\R^n\), then \(\mu = \GenDist\) for some a.s.\ nonnegative random variable \(R^{(n)}\) if and only if there exists a probability measure \(\rho\) on \([0, \infty)\) such that for any nonnegative measurable map \(f \colon \R^n \to \R\) the following polar integration formula is valid,
\begin{equation*}
\int_{\R^n} f \, \dd\mu = \int_{[0, \infty)} \int_{\Sph{p}{n - 1}} f(r \theta) \, \dd\!\Keg_p^{(n)}(\theta) \, \dd\rho(r),
\end{equation*}
and then \(\rho\) is the distribution of \(R^{(n)}\).\\
In particular, \(\mu = \U(\Kug{p}{n})\) if and only if \((R^{(n)})^n \sim \U([0, 1])\).
\item Let \(\mu\) be a probability measure on \(\R^n\) which is absolutely continuous w.r.t.\ \(\Leb^{(n)}\), then \(\mu = \GenDist\) for some a.s.\ nonnegative random variable \(R^{(n)}\) if and only if \(\frac{\dd\mu}{\dd\!\Leb^{(n)}} = g \circ \lVert \cdot \rVert_p\) \(\Leb^{(n)}\)\=/a.e.\ for some nonegative measurable function \(g \colon [0, \infty) \to \R\).
\item Let \(X \sim \GenDist\) for some a.s.\ nonnegative random variable \(R^{(n)}\), then \(\lVert X \rVert_p \GlVert R^{(n)}\); and if \(\prb[X = 0] = 0\), then \(\lVert X \rVert_p\) and \(\frac{X}{\lVert X \rVert_p}\) are independent, and \(\frac{X}{\lVert X \rVert_p} \sim \Keg_p^{(n)}\).
\item Let \(Y = (Y_i)_{i = 1}^n\) be a sequence of i.i.d.\ random variables with \(Y_1 \sim \Gaussp{p}\) and let \(W\) be an a.s.\ nonnegative random variable; define \(R^{(n)} := \frac{\lVert Y \rVert_p}{(\lVert Y \rVert_p^p + W)^{1/p}}\), then \(R^{(n)} \geq 0\) a.s., and \(R^{(n)}\) and \(\frac{Y}{\lVert Y \rVert_p}\) are independent. In particular \(R^{(n)} \frac{Y}{\lVert Y \rVert_p} = \frac{Y}{(\lVert Y \rVert_p^p + W)^{1/p}} \sim \GenDist\).
\end{compactenum}
\end{proposition}

\begin{rem}
\begin{asparaenum}
\item Proposition~\ref{PropGenDistFJP}, 2., tells us that, basically, \(\GenDist\) is constant on \(p\)\=/spheres and thus indeed captures the idea of a measure which is symmetric w.r.t.\ \(\lVert \cdot \rVert_p\).
\item Proposition~\ref{PropGenDistFJP}, 4., shows that \(\GenDist\) is a true generalization of the \(p\)\=/radial measures introduced by Barthe et~al. Note that the latter mainly arise from projecting the cone measure onto the first few coordinates.
\end{asparaenum}
\end{rem}

Finally we remark that if \(\X \sim \GenDist\) and \(\prb[R^{(n)} = 0] > 0\), then also \(\prb[S_{k, p}^{(n)}(\X) = 0] > 0\), and therefore the ratio \(\Ratio{k_1, k_2, p}{(n)}(\X)\) is not defined on a set of positive measure. For that reason we restrict ourselves to the case \(\prb[R^{(n)} = 0] = 0\) in the sequel.

\subsection{Basics from the theory of U\=/statistics}

Since the Maclaurin ratio $ \Ratio{k_1,k_2,p}{(n)}$ can be expressed via suitable U\=/statistics, we need a few results from the theory of U\=/statistics. For \(k, n \in \N\) with \(k \leq n\) define the set \(J_k^{(n)} := \bigl\{ (i_1, \dotsc, i_k) \in \{1, \dotsc, n\}^k \Colon i_1 < \dotsb < i_k \bigr\}\).

\begin{df}
\label{DefUStat}
Let \(k, n \in \N\) with \(k \leq n\), and let \((X_i)_{i \in \N}\) be a sequence of real\-/valued random variables.
\begin{compactenum}
    \item A kernel of order \(k\) is a symmetric measurable function \(\phi^{(k)} \colon \R^k \to \R\).
    \item The U\=/statistic of order \(k\) with the kernel \(\phi^{(k)}\) of order \(k\) is the random variable
    \begin{equation*}
    U_n( \phi^{(k)} ) := U_n( \phi^{(k)} )(X_1, \dotsc, X_n) := \binom{n}{k}^{-1} \sum_{i \in J_k^{(n)}} \phi^{(k)}(X_{i_1}, \dotsc, X_{i_k}).
    \end{equation*}
\end{compactenum}
\end{df}

The following proposition collects a few well-known results (see, e.g., \cite{leeUStatistics,Friedrich1989}).
\begin{proposition}
\label{PropUStat}
Let $(X_i)_{i \in \N}$ be a sequence of i.i.d.\ real-valued random variables, let \(k \in \N\), let $\phi^{(k)} \colon \R^k \rightarrow \R$ be a kernel of order $k$ and, for $n \geq k$, let $U_n( \phi^{(k)})$ be the associated U\=/statistic.
\begin{enumerate}
    \item 
    \label{UStatHdecomp}
    (Hoeffding's decomposition.) If\; $\Erw\left[ \lvert \phi^{(k)} (X_1, \dotsc, X_k )\rvert \right] < \infty $, then \(\Erw[U_n(\phi^{(k)})] = \Erw[\phi^{(k)}(X_1, \dotsc, X_k)] =: \theta\), and we can write
    \begin{equation}
        \label{EqHdecomp}
        U_n( \phi^{(k)}) = \theta + \sum_{h=1}^k \binom{k}{h} U_n( \phi_h),
    \end{equation}
    where the kernels $ \phi_h \colon \R^h \rightarrow \R$ are defined recursively by
    \begin{align*}
    \phi_1(x_1 ) &:= \Erw\bigl[ \phi^{(k)}(x_1, X_2, \dotsc, X_{k}) \bigr] - \theta, \quad \text{and}\\
    \phi_h( x_1, \dotsc, x_h) &:= \Erw\bigl[ \phi^{(k)}(x_1, \dotsc, x_h, X_{h+1}, \dotsc, X_{k} ) \bigr] - \theta - \sum_{c=1}^{h - 1} \sum_{i \in J_c^{(h)}} \phi_c(x_{i_1}, \dotsc, x_{i_c}) \quad \text{for \(h \in \{2, \dotsc, k\}\).}
    \end{align*}
    They satisfy $\Erw[U_n( \phi_h) ]= 0$ for all $h \in \{1, \dotsc, k\}$, and in the case \(\Erw[\phi^{(k)}(X_1, \dotsc, X_k)^2] < \infty\) the U\=/statistics $U_n(\phi_1), \dotsc, U_n(\phi_h)$ are pairwise uncorrelated, and their variances are given by
    \begin{equation*}
        \Var\left[U_n( \phi_h) \right]= \binom{n}{h}^{-1} \Erw\left[ \phi_h(X_1, \ldots, X_h)^2 \right].
    \end{equation*}

    \item 
    \label{USatCLT}
    (Central limit theorem.) Let $m \in \N$ and, for $k_1, \dotsc, k_m \in \N$, let $ \phi^{(1)}, \dotsc, \phi^{(m)}$ be kernels of orders \(k_1, \dotsc, k_m\) resp.\ such that $ \Erw[ \phi^{(i)} (X_1, \dotsc, X_{k_i})^2 ] < \infty$ for all $i \in \{1, \dotsc, m\}$. Let $\theta_{i} := \Erw[ U_n( \phi^{(i)}) ]$, then the following multivariate CLT holds,
    \begin{equation*}
     \left( \sqrt{n} \left( U_n( \phi^{(i)})- \theta_i \right)_{i = 1}^m  \right)_{n \geq \max\{k_1, \dotsc, k_m\}} \KiVert{} \Normal(0, \Sigma).
    \end{equation*}
    The covariance matrix $\Sigma = (\sigma_{i,j})_{i,j = 1}^m \in \R^{m \times m}$ is given by $\sigma_{i,j} = k_i k_j \Erw[ \phi_1^{(i)}(X_1) \phi_1^{(j)}(X_1) ]$, where $\phi_1^{(l)}$ is defined relative to \(\phi^{(l)}\) as in~\ref{UStatHdecomp}, for $l \in \{i, j\}$.
    
    \item
    \label{UStatBE}
    (Berry\--Esseen.) Assume $ \Erw\bigl[ \lvert \phi^{(k)}(X_1, \dotsc, X_k)\rvert^3 \bigr] < \infty$. Then, there exists a constant $C \in \R_+$ such that, for all \(n \in \N\),
    \begin{equation*}
        \dKol\bigl( \sqrt{n} \bigl( U_n( \phi^{(k)}) - \theta \bigr), Z \bigr) \leq \frac{C}{\sqrt{n}},
    \end{equation*}
    where $\theta = \Erw[ \phi^{(k)} ( X_1, \dotsc, X_k) ] $ and $Z \sim \Normal(0, \sigma^2)$ with variance $\sigma^2 := k^2 \Var[\phi_1( X_1)]$, and \(\phi_1\) is defined as in~\ref{UStatHdecomp}.
\end{enumerate}
\end{proposition}

The following propositions establish an MDP and an LDP, respectively, for U\=/statistics with product kernels under more general assumptions than in the classical theory. The proofs are given in Section~\ref{SectionProofsUStat}.

\begin{proposition}
\label{LemMDPUStat}
Let $(X_i)_{i \in \N}$ be an sequence of i.i.d.\ positive random variables such that \(\Erw[X_1] > 0\) and there exists $t_0 \in \R_+$ with $\Erw\bigl[ \exp( t_0 X_1) \bigr] < \infty$. Let $k \in \N$ and define the kernel $\phi^{(k)} \colon \R_+^k \rightarrow \R $ by $(y_1, \dotsc, y_k) \mapsto \prod_{i=1}^k y_i $. Let $(b_n)_{n \in \N} $ be a sequence of positive real numbers such that $ \lim_{n \rightarrow \infty} b_n = \infty $ and $ \lim_{n \rightarrow \infty } \frac{b_n}{\sqrt{n}} = 0$. Then the sequences $\Bigl( \frac{\sqrt{n}}{b_n} \bigl( U_n( \phi^{(k)}) - m^k \bigr) \Bigr)_{n \geq k}$ and $\Bigl( \frac{\sqrt{n}}{b_n} k m^{k-1} \frac{1}{n} \sum_{i =1}^n (X_i - m) \Bigr)_{n \geq k}$ are exponentially equivalent at speed $b_n^2$.

In particular, $\bigl( \frac{\sqrt{n}}{b_n} \bigl( U_{n}(\phi^{(k)}) - m^k \bigr) \bigr)_{n \geq k}$ satisfies an MDP in $\R$ at speed $b_n^2$ with GRF\ $\I \colon \R \rightarrow [0, \infty)$, $x \mapsto \frac{x^2}{2 \sigma^2}$, where $m = \Erw[X_1] $ and $ \sigma^2 = m^{2(k-1)} k^2 \Var[X_1]$.
\end{proposition}

\begin{rem}\label{rem:mdp_ldp_ustat}
\begin{asparaenum}
\item%
Although stated quite generally, we shall need Proposition~\ref{LemMDPUStat} mainly for the case \(X_i = \lvert Y_i \rvert^p\) with \(Y_i \sim \Gaussp{p}\) i.i.d. To the authors' best knowledge, the current general results concerning MDPs and LDPs for U\=/statistics (e.g.~Eichelsbacher and Schmock~\cite{EichSchmo2003} for MDP, Eichelsbacher and L\"owe~\cite{EiLoe1995} for LDP) require the much stricter Cram\'er's condition \(\Erw[\exp(t_0 \lvert \phi(X_1, \dotsc, X_k) \rvert)] < \infty\) for a kernel \(\phi\) of order \(k\), with some \(t_0 > 0\). But as one may verify, for \(Y_i\) as given above there always holds true \(\Erw[\exp(t_0 \lvert Y_1 \rvert^p \dotsm \lvert Y_k \rvert^p)] = \infty\) for any \(t_0 > 0\) and all \(k \geq 2\), yet \(\Erw[\exp(t_0 \lvert Y_1 \rvert^p)] = (1 - p t_0)^{-1/p} < \infty\) for all \(t_0 < \frac{1}{p}\). In that respect the aforementioned propositions establish MDPs and LDPs for certain U\=/statistics based on distributions with somewhat heavier tails than usual.

\item%
To be accurate, there exists kind of an MDP for U\=/statistics with product kernels by Borovskikh and Weber~\cite[Theorem~1]{BorWeb2001}, one might call it ``sharp moderate deviations,'' which supposes the weaker condition \(\Erw[\exp(t_0 \lvert X_1 \rvert)] < \infty\) for some \(t_0 > 0\). (Somewhat misleadingly the title refers to ``large deviations.'') Unfortunately it does not serve our purposes since it only leads to
\begin{equation*}
\lim_{n \to \infty} \frac{1}{b_n^2} \log \prb\left[ \frac{\sqrt{n}}{b_n} \bigl( U_n(\phi^{(k)}) - m^k \bigr) > a \right] = -\frac{a^2}{2 \sigma^2},
\end{equation*}
where notation is as in Proposition~\ref{LemMDPUStat}, and \(a \in \R_+\) is arbitrary. No statement is made about \(a < 0\).

\item%
In the case \(k = 2\) D\"oring and Eichelsbacher~\cite[Theorem~3.1]{DoerEich2013} premised the moment condition \(\Erw[\lvert \phi(X_1, X_2) \rvert^l] \leq C^l l!^\gamma\), for all \(l \geq 3\) and with constants \(\gamma \geq 1\), \(C > 0\) independent of \(l\), in order to prove an MDP for the associated U\=/statistic (albeit with suboptimal speed). It would be very interesting to know whether there exists a generalization of this result to arbitrary \(k\), since it meets precisely our needs (and indeed the condition is satisfied in our situation), and said moment condition seems to be well\-/suited for the investigation of product kernels.
\end{asparaenum}
\end{rem}

\section{Proofs}
\label{SectionProofsPrelim}

\subsection{Proof of Proposition~\ref{PropGenDistFJP}}
\label{ProofGenDistFJP}

\begin{asparaenum}
\item
\(\Rightarrow\): Suppose \(\mu = \GenDist\) for some a.s.\ nonnegative random variable \(R^{(n)}\), then \(\mu\) is the distribution of \(R^{(n)} \Theta\), where \(\Theta \sim \Keg_p^{(n)}\) is independent of \(R^{(n)}\). Write \(\rho\) for the distribution of \(R^{(n)}\), then because of the independence, we have for any nonnegative measurable map \(f \colon \R^n \to \R\),
\begin{equation*}
\int_{\R^n} f \, \dd\mu = \Erw\bigl[ f(R^{(n)} \Theta) \bigr] = \int_{[0, \infty)} \int_{\Sph{p}{n - 1}} f(r \theta) \, \dd\!\Keg_p^{(n)}(\theta) \, \dd\rho(r).
\end{equation*}

\(\Leftarrow\): The polar integration formula implies that \(\mu\) is the distribution of \(R^{(n)} \Theta\), where \(R^{(n)}\) and \(\Theta\) are random variables whose joint distribution is the product measure \(\rho \otimes \Keg_p^{(n)}\). This means that \(R^{(n)} \sim \rho\), \(\Theta \sim \Keg_p^{(n)}\) and the two are independent, and therefore \(\mu = \GenDist\).

\textit{Addendum:} For any nonnegative measurable map \(f \colon \R^n \to \R\) we have the classical polar integration formula (e.g.\ see~\cite{Bartheetal})
\begin{equation*}
\frac{1}{\Leb^{(n)}(\Kug{p}{n})} \int_{\Kug{p}{n}} f \, \dd\!\Leb^{(n)} = n \int_0^1 \int_{\Sph{p}{n - 1}} f(r \theta) r^{n - 1} \, \dd\!\Keg_p^{(n)}(\theta) \, \dd r,
\end{equation*}
so define \(\rho\) by its density \(\frac{\dd\rho}{\dd\!\Leb^{(1)}}(r) = n r^{n - 1}\) if \(r \in [0, 1]\), and zero else, then \(R^{(n)} \sim \rho\) is equivalent to \((R^{(n)})^n \sim \U([0, 1])\).

\item
\(\Rightarrow\): Write \(G := \frac{\dd\mu}{\dd\!\Leb^{(n)}}\), then with a suitable probability measure \(\rho\) on \([0, \infty)\) and using 1.\ we have, for any nonnegative measurable map \(f \colon \R^n \to \R\),
\begin{align*}
\int_{[0, \infty)} \int_{\Sph{p}{n - 1}} f(r \theta) \, \dd\!\Keg_p^{(n)} \, \dd\rho(r) &= \int_{\R^n} f \, \dd\mu = \int_{\R^n} f \, G \, \dd\!\Leb^{(n)}\\
&= n \Leb^{(n)}(\Kug{p}{n}) \int_{[0, \infty)} \int_{\Sph{p}{n - 1}} f(r \theta) G(r \theta) r^{n - 1} \, \dd\!\Keg_p^{(n)}(\theta) \, \dd r.
\end{align*}
Choose \(f := F \circ \lVert \cdot \rVert_p\) with an arbitrary nonnegative measurable map \(F \colon [0, \infty) \to \R\), then \(f(r \theta) = F(r)\) for any \(r \in [0, \infty)\) and \(\theta \in \Sph{p}{n - 1}\), and therewith
\begin{equation*}
\int_{[0, \infty)} F(r) \, \dd\rho(r) = n \Leb^{(n)}(\Kug{p}{n}) \int_{[0, \infty)} F(r) r^{n - 1} \int_{\Sph{p}{n - 1}} G(r \theta) \, \dd\!\Keg_p^{(n)}(\theta) \, \dd r.
\end{equation*}
Define \(g(r) := \int_{\Sph{p}{n - 1}} G(r \theta) \, \dd\!\Keg_p^{(n)}(\theta)\), then the previous display tells us \(\frac{\dd\rho}{\dd\!\Leb^{(1)}}(r) = n \Leb^{(n)}(\Kug{p}{n}) g(r) r^{n - 1}\), and going back to arbitrary \(f\) we have
\begin{align*}
\int_{\R^n} f \, G \, \dd\!\Leb^{(n)} &= \int_{[0, \infty)} \int_{\Sph{p}{n - 1}} f(r \theta) \, \dd\!\Keg_p^{(n)} \, n \Leb^{(n)}(\Kug{p}{n}) g(r) r^{n - 1} \, \dd r\\
&= n \Leb^{(n)}(\Kug{p}{n}) \int_{[0, \infty)} f(r \theta) g(\lVert r \theta \rVert_p) r^{n - 1} \, \dd\!\Keg_p^{(n)}(\theta) \, \dd r\\
&= \int_{\R^n} f(x) g(\lVert x \rVert_p) \, \dd x,
\end{align*}
and from this we obtain \(G = g \circ \lVert \cdot \rVert_p\) \(\Leb^{(n)}\)\=/a.e., as claimed.

\(\Leftarrow\): Let \(f \colon \R^n \to \R\) be nonnegative and measurable. Then,
\begin{align*}
\int_{\R^n} f \, \dd\mu &= \int_{\R^n} f(x) g(\lVert x \rVert_p) \, \dd x\\
&= n \Leb^{(n)}(\Kug{p}{n}) \int_{[0, \infty)} \int_{\Sph{p}{n - 1}} f(r \theta) g(\lVert r \theta \rVert_p) r^{n - 1} \, \dd\!\Keg_p^{(n)}(\theta) \, \dd r\\
&= \int_{[0, \infty)} \int_{\Sph{p}{n - 1}} f(r \theta) \, \dd\!\Keg_p^{(n)}(\theta) \, n \Leb^{(n)}(\Kug{p}{n}) g(r) r^{n - 1} \, \dd r,
\end{align*}
hence the claim follows with \(\frac{\dd\rho}{\dd\!\Leb^{(1)}}(r) := n \Leb^{(n)}(\Kug{p}{n}) g(r) r^{n - 1}\).

\item
By definition \(X \GlVert R^{(n)} \Theta\) with independent random variables \(R^{(n)}\) and \(\Theta \sim \Keg_p^{(n)}\); in particular \(\Theta \in \Sph{p}{n - 1}\) a.s., and hence \(\lVert X \rVert_p \GlVert R^{(n)}\).

Now suppose \(\prb[X = 0] = 0\), then \(\frac{X}{\lVert X \rVert_p}\) is a.s.\ defined. Let \(f \colon [0, \infty) \to \R\) and \(g \colon \Sph{p}{n - 1} \to \R\) be nonnegative and measurable, then, using 1. and denoting the distribution of \(R^{(n)}\) by \(\rho\),
\begin{align*}
\Erw\Biggl[ f(\lVert X \rVert_p) g\Biggl( \frac{X}{\lVert X \rVert_p} \Biggr) \Biggr] &= \int_{[0, \infty)} \int_{\Sph{p}{n - 1}} f(\lVert r \theta \rVert_p) g\Biggl( \frac{r \theta}{\lVert r \theta \rVert_p} \Biggr) \, \dd\!\Keg_p^{(n)}(\theta) \, \dd\rho(r)\\
&= \int_{[0, \infty)} \int_{\Sph{p}{n - 1}} f(r) g(\theta) \, \dd\!\Keg_p^{(n)}(\theta) \, \dd\rho(r)\\
&= \int_{[0, \infty)} f \, \dd\rho \int_{\Sph{p}{n - 1}} g \, \dd\!\Keg_p^{(n)}.
\end{align*}
From that follows the claimed joint distribution of \(\bigl( \lVert X \rVert_p, \frac{X}{\lVert X \rVert_p} \bigr)\).

\item
\(R^{(n)} \geq 0\) a.s.\ is obvious. Write \(\upsilon\) for the distribution of \(W\), and let \(f \colon [0, \infty) \to \R\) and \(g \colon \Sph{p}{n - 1} \to \R\) be nonnegative and measurable, then
\begin{align*}
\Erw\Biggl[ f(R^{(n)}) g\Biggl( \frac{Y}{\lVert Y \rVert_p} \Biggr) \Biggr] &= \int_{[0, \infty)} \int_{\R^n} f\Biggl( \frac{\lVert y \rVert_p}{(\lVert y \rVert_p^p + w)^{1/p}} \Biggr) g\Biggl( \frac{y}{\lVert y \rVert_p} \Biggr) \frac{\ee^{-\lVert y \rVert_p^p/p}}{(2 p^{1/p} \Gamma(\frac{1}{p} + 1))^n} \, \dd y \, \dd\upsilon(w)\\
&= \int_{[0, \infty)} \int_{[0, \infty)} \int_{\Sph{p}{n - 1}} f\Biggl( \frac{\lVert r \theta \rVert_p}{(\lVert r \theta \rVert_p^p + w)^{1/p}} \Biggr) g\Biggl( \frac{r \theta}{\lVert r \theta \rVert_p} \Biggr)\\
&\mspace{150mu} \cdot \frac{n \Leb^{(n)}(\Kug{p}{n}) r^{n - 1} \ee^{-\lVert r \theta \rVert_p^p/p}}{(2 p^{1/p} \Gamma(\frac{1}{p} + 1))^n} \, \dd\!\Keg_p^{(n)}(\theta) \, \dd r \, \dd\upsilon(w)\\
&= \int_{[0, \infty)} \int_{[0, \infty)} f\Biggl( \frac{r}{(r^p + w)^{1/p}} \Biggr) \frac{n \Leb^{(n)}(\Kug{p}{n}) r^{n - 1} \ee^{-r^p/p}}{(2 p^{1/p} \Gamma(\frac{1}{p} + 1))^n} \, \dd r \, \dd\upsilon(w) \int_{\Sph{p}{n - 1}} g \, \dd\!\Keg_p^{(n)},
\end{align*}
and this proves the claimed independence. \qed
\end{asparaenum}

The last integral may be exploited to determine the distribution of \(R^{(n)}\): separate the possible atom \(w = 0\), perform the substitution \(s := \frac{r}{(r^p + w)^{1/p}}\) for any fixed \(w \in \R_+\), then after simplifying terms, including \(\Leb^{(n)}(\Kug{p}{n}) = \frac{(2 \Gamma(1 + \frac{1}{p}))^n}{\Gamma(1 + \frac{n}{p})}\), we reach
\begin{equation*}
\Erw[f(R^{(n)})] = f(1) \upsilon\{0\} + \int_0^1 f(s) \frac{1}{p^{n/p - 1} \Gamma(\frac{n}{p})} \int_{\R_+} w^{n/p} \exp\Biggl( -\frac{s^p w}{p (1 - s^p)} \Biggr) \, \dd\upsilon(w) \, \frac{s^{n - 1}}{(1 - s^p)^{n/p + 1}} \, \dd s,
\end{equation*}
so \(\rho\) is supported in \([0, 1]\) and absolutely continuous w.r.t.\ \(\Leb^{(1)}\) plus a point mass of weight \(\upsilon\{0\}\) at 1; of course this is compatible essentially with the expressions given in~\cite{Bartheetal} and the literature building thereupon.

\subsection{Proof of Proposition~\ref{LemMDPUStat}}
\label{SectionProofsUStat}

The following lemma is used in the proofs of Proposition~\ref{LemMDPUStat} and Theorem~\ref{ThmMDPRkn}.

\begin{lem}
\label{PropExpEquivalence}
Let $ d \in \N$, let $R \colon \R^d \rightarrow \R$ be a continuous function, and let $M, \delta \in \R_+$ such that $ \lvert R(x)\rvert \leq M \lVert x\rVert_2^2$ for all $x \in \R^d$ with $ \lVert x\rVert_2^2 \leq \delta$; let $(b_n)_{n \in \N}$ be a sequence of positive real number such that $ \lim_{n \to \infty} b_n = \infty$ and $ \lim_{n \to \infty} \frac{b_n}{\sqrt{n}} = 0$, and for each $i \in \{1, \dotsc, d\}$ let $ (Z_n^{(i)})_{n \in \N}$ be a sequence of random variables such that $ \left( \frac{\sqrt{n}}{b_n} Z_n^{(i)} \right)_{n \in \N}$ satisfies an MDP at speed $b_n^2$ with GRF $x \mapsto \frac{x^2}{2 \sigma_i^2}$, where $ \sigma_i^2 \in \R_+$. Then $\left( \frac{\sqrt{n}}{b_n} \, R(Z_n^{(1)}, \dotsc, Z_n^{(d)}) \right)_{n \in \N}$ is exponentially equivalent to $0$ at speed $b_n^2$. 
\end{lem}

\begin{proof}
We need to show that, for any $\epsilon \in \R_+$,
\begin{equation*}
    \lim_{n \rightarrow \infty} \frac{1}{b_n^2} \log \left( \prb\left[ \frac{\sqrt{n}}{b_n} \left\lvert  R \left (Z_n^{(1)}, \dotsc, Z_n^{(d)} \right) \right\rvert > \epsilon \right] \right) = - \infty. 
\end{equation*}
Let \(\epsilon \in \R_+\). We take a closer look at the probability in the previous expression, where we get, for all $n \in \N$ sufficiently large,
\begin{align*}
\prb\left[ \frac{\sqrt{n}}{b_n} \left\lvert R \left (Z_n^{(1)}, \dotsc, Z_n^{(d)} \right) \right\rvert > \epsilon \right]  & =
\prb\left[ \frac{\sqrt{n}}{b_n} \left\lvert R \left (Z_n^{(1)}, \dotsc, Z_n^{(d)} \right) \right\rvert > \epsilon, \left\lVert \left (Z_n^{(1)}, \dotsc, Z_n^{(d)} \right) \right\rVert_2^2 \leq \delta \right] \\
& \quad + 
\prb\left[ \frac{\sqrt{n}}{b_n} \left\lvert R \left (Z_n^{(1)}, \dotsc, Z_n^{(d)} \right) \right\rvert > \epsilon, \left\lVert \left (Z_n^{(1)}, \dotsc, Z_n^{(d)} \right) \right\rVert_2^2 > \delta \right] \\
& \leq \prb\left[ \left\lVert \left (Z_n^{(1)}, \dotsc, Z_n^{(d)} \right) \right\rVert_2^2  > \frac{\epsilon b_n}{M \sqrt{n}}  \right]
+
\prb\left[ \left\lVert \left (Z_n^{(1)}, \dotsc, Z_n^{(d)} \right) \right\rVert_2^2 > \delta \right] \\
& \leq 2 \prb\left[ \left\lVert \left (Z_n^{(1)}, \dotsc, Z_n^{(d)} \right) \right\rVert_2^2  > \frac{\epsilon b_n}{M \sqrt{n}}  \right].
\end{align*}
Here we have used the local properties of $R$. The estimate in the last inequality holds since $\delta > \frac{\epsilon b_n}{M \sqrt{n}} $ if $n$ is sufficiently large. We further get
\begin{align*}
    \prb\left[ \left\lVert \left (Z_n^{(1)}, \dotsc, Z_n^{(d)} \right) \right\rVert_2^2  > \frac{\epsilon b_n}{ M \sqrt{n}}  \right]
    & \leq 
    \sum_{i=1}^d \prb\left[  \left( Z_n^{(i)} \right )^2  > \frac{\epsilon b_n}{ d M \sqrt{n}} \right] \\
    & \leq d \max_{i\in \{1, \dotsc, d\}} \prb\left[  \left( Z_n^{(i)} \right )^2  > \frac{\epsilon b_n}{ d M \sqrt{n}}  \right] \\
    & = d \max_{i\in \{1, \dotsc, d\}} \prb\left[  \left( \frac{\sqrt{n}}{b_n} Z_n^{(i)} \right )^2  > \frac{\epsilon \sqrt{n} }{ d M b_n}  \right].
\end{align*}
Let $T \in /R_+$, then, for sufficiently large $n \in \N$, we have $\frac{\epsilon \sqrt{n} }{ d M b_n} > T $, and thus, for each \(i \in \{1, \dotsc, d\}\),
\begin{equation*}
    \prb\left[ \left( \frac{\sqrt{n}}{b_n} Z_n^{(i)} \right )^2 > \frac{\epsilon \sqrt{n} }{ d M b_n}  \right] \leq 
    \prb\left[ \left( \frac{\sqrt{n}}{b_n} Z_n^{(i)} \right )^2 > T \right].
\end{equation*}
Since, for each $i \in \{1, \dotsc, d\}$, $ \left( \frac{\sqrt{n}}{b_n} Z_n^{(i)}\right)_{n \in \N}$ satisfies an MDP in $\R$ at speed $b_n^2$ with GRF $ x \mapsto \frac{x^2}{2 \sigma_i^2}$, we get
\begin{equation*}
\limsup_{n \to \infty} \frac{1}{b_n^2} \log\Biggl( \prb\Biggl[ \Biggl( \frac{\sqrt{n}}{b_n} \, Z_n^{(i)} \Biggr)^2 > T \Biggr] \Biggr) \leq -\inf_{x^2 \geq T} \frac{x^2}{2 \sigma_i^2} = -\frac{T}{2 \sigma_i^2}.
\end{equation*}
This leads to
\begin{align*}
 \limsup_{n \rightarrow \infty} \frac{1}{b_n^2} \log \left( \prb\left[ \frac{\sqrt{n}}{b_n} \left\lvert R \left (Z_n^{(1)}, \dotsc, Z_n^{(d)} \right) \right\rvert > \epsilon \right] \right) 
 & \leq 
 \max_{i\in \{1, \dotsc, d\}} \limsup_{n \rightarrow \infty} \frac{1}{b_n^2} \log \prb\left[  \left( \frac{\sqrt{n}}{b_n} Z_n^{(i)} \right )^2  > T  \right] \\
 & \leq \max_{i\in \{1, \dotsc, d\}} \Biggl( -\frac{T}{2 \sigma_i^2} \Biggr)\\
 & = -\frac{T}{2} \min_{i\in \{1, \dotsc, d\}}  \sigma_i^{-2}.
\end{align*}
By taking the limit $T \rightarrow \infty$, we establish the claimed exponential equivalence at speed $b_n^2$. 
\end{proof}

The following lemma is of a combinatorial nature; its only use is in the proof of Proposition~\ref{LemMDPUStat} given below.

\begin{lem}\label{LemSummen}
Let \(k, n \in \N\) with \(2 \leq k \leq n\), denote by \(B_k^{(n)}\) the set of \(k\)\=/tuples of integers in \([1, n]\) such that at least two components are equal, and let \(x_1, \dotsc, x_n \in [0, \infty)\). Then,
\begin{equation*}
\sum_{i \in B_k^{(n)}} x_{i_1} \dotsm x_{i_k} \leq \binom{k}{2} \sum_{i = 1}^n x_i^2 \Biggl( \sum_{i = 1}^n x_i \Biggr)^{k - 2}.
\end{equation*}
\end{lem}

\begin{proof}
Call \(C_{n, k} := \{\alpha \in \N_0^n \Colon \lvert \alpha \rvert = k \wedge \exists i \in \{1, \dotsc, n\} \colon \alpha_i \geq 2\}\), the set of multiindices which sum to \(k\) and of which at least one component is no smaller than \(2\), then we can write (recall the notations \(x^\alpha = \prod_{i = 1}^n x_i^{\alpha_i}\) and \(\binom{k}{\alpha} = k! \bigl( \prod_{i = 1}^n \alpha_i! \bigr)^{-1}\))
\begin{equation*}
\sum_{i \in B_k^{(n)}} x_{i_1} \dotsm x_{i_k} = \sum_{\alpha \in C_{n, k}} \binom{k}{\alpha} x^\alpha.
\end{equation*}
Using the multinomial theorem we get
\begin{equation}\label{EqMultinom}
\sum_{i = 1}^n x_i^2 \Biggl( \sum_{i = 1}^n x_i \Biggr)^{k - 2} = \sum_{i = 1}^n \sum_{\substack{\alpha \in \N_0^n \\ \lvert \alpha \rvert = k - 2}} \binom{k - 2}{\alpha} x_i^2 x^\alpha.
\end{equation}
The main task is now to count how often a given monomial \(x^\alpha\), for \(\alpha \in C_{n, k}\), appears on the right\-/hand side of~\eqref{EqMultinom}. Let \(\alpha \in C_{n, k}\), then \(m := \lvert \{ i \in \{1, \dotsc, n\} \Colon \alpha_i \geq 2\} \rvert \geq 1\), and let \(i_1, \dotsc, i_m \in \{1, \dotsc, n\}\) be an enumeration of the indices for which \(\alpha_i \geq 2\). Then \(x^\alpha = x_{i_j}^2 x^{\alpha - 2 e_{i_j}}\) for any \(j \in \{1, \dotsc, m\}\) (here \(e_i\) denotes the \(i\)\textsuperscript{th} canonical unit vector in \(\R^n\)), and the latter term occurs \(\binom{k - 2}{\alpha - 2 e_{i_j}}\) times on the right\-/hand side of~\eqref{EqMultinom}, and those are all occurrences of \(x^\alpha\) on said right\-/hand side. Thus the coefficient of \(x^\alpha\) in~\eqref{EqMultinom} equals
\begin{align*}
\sum_{j = 1}^m \binom{k - 2}{\alpha - 2 e_{i_j}} &= \sum_{j = 1}^m \frac{(k - 2)!}{\alpha_1! \dotsm \alpha_{i_j - 1}! (\alpha_{i_j} - 2)! \alpha_{i_j + 1}! \dotsm \alpha_n!}\\
&= \sum_{j = 1}^m \binom{k}{\alpha} \frac{\alpha_{i_j} (\alpha_{i_j} - 1)}{k (k - 1)}\\
&= \binom{k}{\alpha} \binom{k}{2}^{-1} \sum_{j = 1}^m \binom{\alpha_{i_j}}{2}.
\end{align*}
It remains to notice that \(m \geq 1\) and \(\binom{\alpha_{i_1}}{2} \geq 1\), and the claim follows.
\end{proof}

\begin{proof}[Proof of Proposition~\ref{LemMDPUStat}]
We recall that, by definition, $U_n( \phi^{(k)}) = \binom{n}{k}^{-1} \sum_{i \in J_k^{(n)}} X_{i_1} \dotsm X_{i_k}$. We show the result in the case of $k\geq 2$, since it is immediately true for $k=1$. For clarity's sake we are going to break down into several steps the proof that \(\bigl( \frac{\sqrt{n}}{b_n} (U_n(\phi^{(k)}) - m^k) \bigr)_{n \geq k}\) and \(\bigl( \frac{\sqrt{n}}{b_n} \, \frac{k m^{k - 1}}{n} \sum_{i = 1}^n (X_i - m) \bigr)_{n \geq k}\) are exponentially equivalent at speed \(b_n^2\), as follows:
\begin{compactenum}
\item The following are exponentially equivalent at speed \(b_n^2\):
    \begin{equation*}
    \Biggl( \frac{\sqrt{n}}{b_n} \Biggl( \frac{1}{n} \sum_{i = 1}^n X_i \Biggr)^k \Biggr)_{n \geq k} \quad \text{and} \quad \Biggl( \frac{\sqrt{n}}{b_n} \, \frac{k!}{n^k} \sum_{i \in J_k^{(n)}} X_{i_1} \dotsm X_{i_k} \Biggr)_{n \geq k}.
    \end{equation*}
\item The following are exponentially equivalent at speed \(b_n^2\):
    \begin{equation*}
    \Biggl( \frac{\sqrt{n}}{b_n} \, \frac{k!}{n^k} \sum_{i \in J_k^{(n)}} X_{i_1} \dotsm X_{i_k} \Biggr)_{n \geq k} \quad \text{and} \quad \Biggl( \frac{\sqrt{n}}{b_n} U_n(\phi^{(k)}) \Biggr)_{n \geq k}.
    \end{equation*}
\item The following are exponentially equivalent at speed \(b_n^2\):
    \begin{equation*}
    \Biggl( \frac{\sqrt{n}}{b_n} \bigl( U_n(\phi^{(k)}) - m^k \bigr) \Biggr)_{n \geq k} \quad \text{and} \quad \Biggl( \frac{\sqrt{n}}{b_n} \Biggl( \Biggl( \frac{1}{n} \sum_{i = 1}^n X_i \Biggr)^k - m^k \Biggr) \Biggr)_{n \geq k}.
    \end{equation*}
\item The following are exponentially equivalent at speed \(b_n^2\):
    \begin{equation*}
    \Biggl( \frac{\sqrt{n}}{b_n} \Biggl( \Biggl( \frac{1}{n} \sum_{i = 1}^n X_i \Biggr)^k - m^k \Biggr) \Biggr)_{n \geq k} \quad \text{and} \quad \Biggl( \frac{\sqrt{n}}{b_n} \, \frac{k m^{k - 1}}{n} \sum_{i = 1}^n (X_i - m) \Biggr)_{n \geq k}.
    \end{equation*}
\end{compactenum}
Steps~3 and~4 imply the claim, as exponential equivalence is an equivalence relation indeed and, hence, transitive.

The MDP for \(\bigl( \frac{\sqrt{n}}{b_n} (U_n(\phi^{(k)}) - m^k) \bigr)_{n \geq k}\) is an easy consequence then: by assumption \(\Erw[\exp(t_0 X_1)] < \infty\) for some \(t_0 \in \R_+\), hence also \(\Erw[\exp(t_1 k m^{k - 1} X_1)] < \infty\) for \(t_1 = \frac{t_0}{k m^{k - 1}}\), and therefore by Proposition~\ref{ThmCramerMDP}, \(\Bigl( \frac{\sqrt{n}}{b_n} \, \frac{1}{n} \sum_{i = 1}^n k m^{k - 1} (X_i - m) \Bigr)_{n \geq k}\) satisfies an MDP at speed \(b_n^2\) with GRF \(x \mapsto \frac{x^2}{2 \Var[k m^{k - 1} X_1]} = \frac{x^2}{2 \sigma^2}\). Thus by Proposition~\ref{PropExpEquLDP} and the exponential equivalence established before, \(\bigl( \frac{\sqrt{n}}{b_n} (U_n(\phi^{(k)}) - m^k) \bigr)_{n \geq k}\) satisfies the same MDP.

\textit{Step~1.} On the one hand we can write $\frac{k!}{n^k} \sum_{i \in J_k^{(n)}} X_{i_1} \dotsm X_{i_k} = n^{-k} \sum_{i \in I_{k}^{(n)} } X_{i_1} \dotsm X_{i_k}$, where $I_k^{(n)}$ denotes the set of $k$\=/tuples of pairwise distinct indices from \(\{1, \dotsc, n\}\). On the other hand we see $\bigl( \frac{1}{n} \sum_{i=1}^n X_i \bigr)^k = n^{-k} \sum_{i_1, \dotsc, i_k=1}^n X_{i_1} \dotsm X_{i_k}$. Thus we get
\begin{equation*}
\frac{k!}{n^k} \sum_{i \in J_k^{(n)}} X_{i_1} \dotsm X_{i_k} - \Biggl( \frac{1}{n} \sum_{i=1}^n X_i \Biggr)^k = n^{-k} \sum_{i \in B_{k}^{(n)}} X_{i_1} \dotsm X_{i_k},
\end{equation*}
where $B_{k}^{(n)}$ is defined as in Lemma~\ref{LemSummen}, that is, \(B_k^{(n)} = \{1, \dotsc, n\}^k \setminus I_k^{(n)}\). Employing said lemma we get

\begin{equation*}
0 \leq n^{-k} \sum_{i \in B_{k}^{(n)}} X_{i_1} \dotsm X_{i_k} \leq \frac{1}{n} \binom{k}{2} \frac{1}{n} \sum_{i=1}^n X_i^2 \Biggl( \frac{1}{n} \sum_{i=1}^n X_i \Biggr)^{k-2}.
\end{equation*}
The claimed exponential equivalence means that we need to show, for any $\epsilon \in \R_+$,
\begin{equation*}
\lim_{n \rightarrow \infty} \frac{1}{b_n^2} \log \left( \prb\left[ \left\lvert \frac{\sqrt{n}}{b_n} \, \frac{k!}{n^k} \sum_{i \in J_k^{(n)}} X_{i_1} \dotsm X_{i_k} - \frac{\sqrt{n}}{b_n} \Biggl( \frac{1}{n} \sum_{i=1}^n X_i \Biggr)^k \right\rvert > \epsilon \right] \right) = -\infty.
\end{equation*}
Let \(\epsilon > 0\). By the previous calculation it suffices to show
\begin{equation*}
\lim_{n \rightarrow \infty} \frac{1}{b_n^2} \log \left( \prb\left[
\frac{1}{b_n \sqrt{n}}
\frac{1}{n} \sum_{i=1}^n X_i^2 \left( \frac{1}{n} \sum_{i=1}^n X_i \right)^{k-2}
> \epsilon  
\right] \right) =- \infty.
\end{equation*}
We take a closer look at the probability in the logarithm, where we get, for any $n \in \N$ and $M \in \R_+$,
\begin{align*}
\prb\left[ \frac{1}{\sqrt{n}b_n} \frac{1}{n} \sum_{i=1}^n X_i^2 \left( \frac{1}{n} \sum_{i=1}^n X_i \right)^{k-2} > \epsilon \right] & = 
\prb\left[ \frac{1}{\sqrt{n}b_n} \frac{1}{n} \sum_{i=1}^n X_i^2 \left( \frac{1}{n} \sum_{i=1}^n X_i \right)^{k-2} > \epsilon, \left( \frac{1}{n} \sum_{i=1}^n X_i \right)^{k-2} > M \right]\\
& \quad + 
\prb\left[ \frac{1}{\sqrt{n}b_n} \frac{1}{n} \sum_{i=1}^n X_i^2 \left( \frac{1}{n} \sum_{i=1}^n X_i \right)^{k-2} > \epsilon, \left( \frac{1}{n} \sum_{i=1}^n X_i \right)^{k-2} \leq M \right]\\
& \leq  
\prb\Biggl[ \Biggl( \frac{1}{n} \sum_{i=1}^n X_i \Biggr)^{k-2} > M \Biggr]
+
\prb\Biggl[ \frac{1}{\sqrt{n}b_n} \frac{1}{n} \sum_{i=1}^n X_i^2 > \frac{\epsilon}{M} \Biggr]\\
& =: \alpha_n + \beta_n.
\end{align*}
Now it suffices to prove \(\limsup_{n \to \infty} \frac{1}{b_n^2} \log(\alpha_n) = -\infty\) and \(\limsup_{n \to \infty} \frac{1}{b_n^2} \log(\beta_n) = -\infty\), because then
\begin{equation}\label{EqLogSumme}
\limsup_{n \to \infty} \frac{1}{b_n^2} \log(\alpha_n + \beta_n) = \sup\Biggl\{ \limsup_{n \to \infty} \frac{1}{b_n^2} \log(\alpha_n), \limsup_{n \to \infty} \frac{1}{b_n^2} \log(\beta_n) \Biggr\} = -\infty.
\end{equation}
Since there exists a $ t_0 \in \R_+$ with $ \Erw[ \exp( t_0 X_1) ] < \infty $ by assumption, we know from Theorem~\ref{ThmCramer} that $ (X_i)_{i \in \N}$ satisfies an LDP in $\R$ at speed $n$ with GRF $ \I \colon \R \rightarrow [0, \infty] $ which is the Legendre transform of the cumulant generating function of $X_1$. Therefore,
\begin{equation*}
\limsup_{n \to \infty} \frac{1}{n} \log(\alpha_n) \leq -\inf_{x \geq M^{1/(k - 2)}} \I(x).
\end{equation*}
If $M \in \R_+$ is sufficiently large, to wit $ M^{1/(k-2)} > \Erw[X_1] $, then by the properties of $\I$, $\inf_{x \geq M^{1/(k-2)} } \I(x) = \I(M^{1/(k - 2)}) > 0$. As \(\lim_{n \to \infty} \frac{n}{b_n^2} = \infty\), we obtain
\begin{equation*}
\limsup_{n \to \infty} \frac{1}{b_n^2} \log(\alpha_n) = \limsup_{n \to \infty} \frac{n}{b_n^2} \, \frac{1}{n} \log(\alpha_n) = -\infty.
\end{equation*}

Now we turn to $(\beta_n)_{n \in \N}$. By assumption, we know that there exists a $t_0 \in \R_+$ with $\Erw[\exp(t_0 X_1)] < \infty$ and hence, by Markov's inequality, we find constants $b, c \in \R_+$ such that $ \prb[ X_1 \geq t] \leq c \exp(-b t)$, for $t \in \R_+$. Thus we have $\prb[X_1^2 \geq t] \leq c \exp(-b t^{1/2})$ for any $t \in \R_+$. First we assume $\liminf_{n \rightarrow \infty} \frac{\sqrt{n}}{b_n^2} > 0$. Since $\lim_{n \to \infty} \sqrt{n} b_n = \infty$, for any $A \in (\Erw[X_1^2], \infty)$ there exists \(n_0 \geq k\) such that we have $\frac{ \sqrt{n}b_n \epsilon}{M} \geq A$ for all \(n \geq n_0\), and hence, for all \(n \geq n_0\),
\begin{equation*}
\beta_n = \prb\left[ \frac{1}{\sqrt{n}b_n} \, \frac{1}{n} \sum_{i=1}^n X_i^2 > \frac{\epsilon}{M} \right ] = \prb\left[ \frac{1}{n} \sum_{i=1}^n X_i^2 > \frac{ \sqrt{n}b_n \epsilon}{M} \right] \leq \prb\left[ \frac{1}{n} \sum_{i=1}^n X_i^2 > A \right].
\end{equation*}
By Proposition~\ref{LemStrechedExp}, we get
\begin{equation*}
\limsup_{n \to \infty} \frac{1}{\sqrt{n}} \log\Biggl( \prb\Biggl[ \frac{1}{n} \sum_{i = 1}^n X_i^2 > A \Biggr] \Biggr) \leq -b (A - \Erw[X_1^2])^{1/2},
\end{equation*}
and hence, since \(A\) is arbitrary,
\begin{equation*}
\limsup_{n \to \infty} \frac{1}{\sqrt{n}} \log\Biggl( \prb\Biggl[ \frac{1}{n} \sum_{i = 1}^n X_i^2 > A \Biggr] \Biggr) = -\infty.
\end{equation*}
This implies
\begin{equation*}
\limsup_{n \rightarrow \infty} \frac{1}{b_n^2} \log(\beta_n) \leq \limsup_{n \rightarrow \infty} \frac{\sqrt{n}}{b_n^2} \, \frac{1}{\sqrt{n}} \log\Biggl( \prb\Biggl[ \frac{1}{n} \sum_{i = 1}^n X_i^2 > A \Biggr] \Biggr) = -\infty.
\end{equation*}
Now we consider the case $\liminf_{n \rightarrow \infty} \frac{\sqrt{n}}{b_n^2} = 0$; and w.l.o.g.\ we can assume $\lim_{n \rightarrow \infty} \frac{\sqrt{n}}{b_n^2} = 0$, otherwise switch to a corresponding subsequence (all other accumulation points have been dealt with in the previous case). Then there exists \(n_1 \geq k\) such that \(b_n^4 \geq n\) for all \(n \geq n_1\), and then \(\sum_{i = 1}^n X_i^2 \leq \sum_{i = 1}^{b_n^4} X_i^2\). Moreover, because we still have \(\lim_{n \to \infty} \frac{\sqrt{n}}{b_n} = \infty\), for any \(A \in (\Erw[X_1^2], \infty)\) there exists \(n_2 \geq k\) such that \(\frac{\epsilon n \sqrt{n}}{b_n^3} \geq A\) for all \(n \geq n_2\). This yields, for all \(n \geq \max\{n_1, n_2\}\),
\begin{equation*}
\beta_n = \prb\left[ \frac{1}{\sqrt{n}b_n} \, \frac{1}{n} \sum_{i=1}^n X_i^2 > \frac{\epsilon}{M} \right ] \leq \prb\Biggl[ \frac{1}{b_n^4} \sum_{i = 1}^{b_n^4} X_i^2 > A \Bigr],
\end{equation*}
and therefore,
\begin{align*}
\limsup_{n \rightarrow \infty} \frac{1}{b_n^2} \log(\beta_n)
\leq \limsup_{n \rightarrow \infty} \frac{1}{b_n^2} \log\Biggl( \prb\Biggl[ \frac{1}{b_n^4} \sum_{i=1}^{b_n^4} X_i^2 > A \Biggr] \Biggr)
\leq -b (A - \Erw[X_1^2])^{1/2},
\end{align*}
where again we have used Proposition~\ref{LemStrechedExp} in the last step. Now by taking the limit $A \rightarrow \infty$, we get the claimed exponential equivalence.
 
\textit{Step~2.} We now show that $\bigl( \frac{\sqrt{n}}{b_n} \frac{k!}{n^k} \sum_{i \in J_k^{(n)}} X_{i_1} \dotsm X_{i_k} \bigr)_{n \geq k}$ and $\bigl( \frac{\sqrt{n}}{b_n} U_n(\phi^{(k)}) \bigr)_{n \geq k}$ are exponentially equivalent at speed $(b_n^2)_{n \in \N}$. Let $\delta_n := \frac{n^k}{k!} \binom{n}{k}^{-1} - 1$ and note $n \delta_n \to \frac{k (k - 1)}{2}$ as $n \to \infty$, and hence $\frac{\sqrt{n}}{b_n} \delta_n \to 0$ for any sequence $(b_n)_{n \in \N}$ with $\lim_{n \to \infty} b_n = \infty$. For $\epsilon \in \R_+$ we get
\begin{align*}
\prb\Biggl[ \frac{\sqrt{n}}{b_n} \Biggl\lvert U_n(\phi^{(k)}) - \frac{k!}{n^k} \sum_{i \in J_k^{(n)}} X_{i_1} \dotsm X_{i_k} \Biggr\rvert > \epsilon \Biggr] & = 
\prb\Biggl[ \frac{\sqrt{n}}{b_n} \delta_n \Biggl\lvert \frac{k!}{n^k} \sum_{i \in J_k^{(n)}} X_{i_1} \dotsm X_{i_k} \Biggr\rvert > \epsilon \Biggr] \\
& \leq 
\prb\Biggl[ \frac{\sqrt{n}}{b_n} \delta_n \Biggl\lvert \frac{k!}{n^k} \sum_{i \in J_k^{(n)}} X_{i_1} \dotsm X_{i_k}  - \left( \frac{1}{n} \sum_{i=1}^n X_i \right)^k \Biggr\rvert > \frac{\epsilon}{2} \Biggr]\\ 
& \quad +
\prb\Biggl[ \frac{\sqrt{n}}{b_n} \delta_n \Biggl( \frac{1}{n} \sum_{i=1}^n X_i \Biggr)^k > \frac{\epsilon}{2} \Biggr] \\
& =: \alpha_n + \beta_n.
\end{align*}
Again by Equation~\eqref{EqLogSumme} it suffices to show $\lim_{n \rightarrow \infty} \frac{1}{b_n^2} \log(\alpha_n) = -\infty$ and $\lim_{n \rightarrow \infty} \frac{1}{b_n^2} \log(\beta_n) = -\infty$. We start with the sequence $(\beta_n)_{n \geq k}$. Let $A \in (\Erw[X_1], \infty)$ and note that $\frac{\epsilon b_n }{2 \sqrt{n} \delta_n } \geq A^k$ for all $n \in \N$ sufficiently large. We get
 \begin{align*}
 \limsup_{n \rightarrow \infty} \frac{1}{b_n^2 } \log( \beta_n ) & = \limsup_{n \rightarrow \infty} \frac{1}{b_n^2} \log\Biggl( \prb\Biggl[ \frac{\sqrt{n}}{b_n} \delta_n \Biggl( \frac{1}{n} \sum_{i=1}^n X_i \Biggr)^k > \frac{\epsilon}{2} \Biggr] \Biggr) \\
 & \leq \limsup_{n \rightarrow \infty} \frac{1}{b_n^2 } \log\Biggl( \prb\Biggl[ \frac{1}{n} \sum_{i=1}^n X_i > A \Biggr] \Biggr) \\
 & \leq - \limsup_{n \rightarrow \infty}  \frac{n}{b_n^2} \inf_{x \geq A} \I(x) \\
 & = - \infty
 \end{align*}
 where we have used the large deviations upper bound for the sequence $\left( \frac{1}{n} \sum_{i=1}^n X_i \right)_{n \in \N}$ together with $ \inf_{x \geq A} \I(x) = \I(A) \in \R_+$, since $A > \Erw[X_1]$. For the sequence $(\alpha_n)_{n \geq k}$ we get
 \begin{align*}
\limsup_{n \rightarrow \infty} \frac{1}{b_n^2} \log(\alpha_n) 
 & = 
 \limsup_{n \rightarrow \infty} \frac{1}{b_n^2} \log\Biggl(
 \prb\Biggl[ \frac{\sqrt{n}}{b_n} \delta_n \Biggl\lvert \frac{k!}{n^k} \sum_{i \in J_k^{(n)}} X_{i_1} \dotsm X_{i_k} - \Biggl( \frac{1}{n} \sum_{i=1}^n X_i \Biggr)^k \Biggr\rvert > \frac{\epsilon}{2} \Biggr] \Biggr) \\
 & \leq 
 \limsup_{n \rightarrow \infty} \frac{1}{b_n^2} \log\Biggl(
 \prb\Biggl[ \frac{\sqrt{n}}{b_n} \Biggl\lvert \frac{k!}{n^k} \sum_{i \in J_k^{(n)}} X_{i_1} \dotsm X_{i_k} - \Biggl( \frac{1}{n} \sum_{i=1}^n X_i \Biggr)^k \Biggr\rvert > \epsilon  \Biggr] \Biggr) \\
 & = - \infty,
 \end{align*}
where we have used $\frac{\epsilon}{ 2 \delta_n} > \epsilon $ for all $n \in \N$ sufficiently large, and have employed the exponential equivalence at scale $(b_n^2)_{n \in \N}$ of $\bigl( \frac{\sqrt{n}}{b_n} \frac{k!}{n^k} \sum_{i \in J_k^{(n)}} X_{i_1} \dotsm X_{i_k} \bigr)_{n \geq k}$ and $\bigl( \frac{\sqrt{n}}{b_n} \bigl( \frac{1}{n} \sum_{i=1}^n X_i \bigr)^k \bigr)_{n \geq k}$ from Step~1.

\textit{Step~3.} From Steps~1 and~2 and by transitivity of exponential equivalence we know now that \(\bigl( \frac{\sqrt{n}}{b_n} \bigl( \frac{1}{k} \sum_{i = 1}^n X_i \bigr)^k \bigr)_{n \geq k}\) and \(\bigl( \frac{\sqrt{n}}{b_n} \, U_n(\phi^{(k)}) \bigr)_{n \geq n}\) are exponentially equivalent at speed \(b_n^2\). But since trivially
\begin{equation*}
\frac{\sqrt{n}}{b_n} \Biggl( \Biggl( \frac{1}{k} \sum_{i = 1}^n X_i \Biggr)^k - m^k \Biggr) - \frac{\sqrt{n}}{b_n} \bigl( U_n(\phi^{(k)}) - m^k \bigr) = \frac{\sqrt{n}}{b_n} \Biggl( \frac{1}{k} \sum_{i = 1}^n X_i \Biggr)^k - \frac{\sqrt{n}}{b_n} \, U_n(\phi^{(k)}),
\end{equation*}
this immediately implies that also \(\bigl( \frac{\sqrt{n}}{b_n} \bigl( \bigl( \frac{1}{k} \sum_{i = 1}^n X_i \bigr)^k - m^k \bigr) \bigr)_{n \geq k}\) and \(\bigl( \frac{\sqrt{n}}{b_n} (U_n(\phi^{(k)}) - m^k) \bigr)_{n \geq n}\) are exponentially equivalent at speed \(b_n^2\).
 
\textit{Step~4.} The function $x \mapsto x^k$ has the Taylor expansion $x^k = m^k +k m^{k-1} (x - m) + R(x - m)$ around $m$, where there exist $M, \delta \in \R_+$ with $\lvert R(x) \rvert \leq M x^2$ whenever $x^2 < \delta$. This implies that
\begin{equation*}
\frac{\sqrt{n}}{b_n} \left[ \Biggl( \frac{1}{n} \sum_{i=1}^n X_i \Biggr)^k - m^k - k m^{k-1} \Biggl( \frac{1}{n} \sum_{i=1}^n X_i -m \Biggr) \right] = \frac{\sqrt{n}}{b_n} R\Biggl( \frac{1}{n} \sum_{i=1}^n X_i -m \Biggr).
\end{equation*}
Since $\bigl( \frac{1}{\sqrt{n} b_n} \sum_{i=1}^n (X_i - m) \bigr)_{n \in \N}$ satisfies an MDP at speed $b_n^2$ with GRF $x \mapsto \frac{x^2}{2 \Var[X_1]}$, we can apply Lemma~\ref{PropExpEquivalence} to infer that $\bigl( \frac{\sqrt{n}}{b_n} R\bigl( \frac{1}{n} \sum_{i=1}^n X_i - m \bigr) \bigr)_{n \in \N}$ is exponentially equivalent to $0$ at speed $b_n^2$. Thus the sequences $\Bigl( \frac{\sqrt{n}}{b_n} \Bigl( \bigl( \frac{1}{n} \sum_{i=1}^n X_i \bigr)^k - m^k \Bigr) \Bigr)_{n \in \N}$ and $\Bigl( \frac{\sqrt{n}}{b_n} k m^{k-1} \frac{1}{n} \sum_{i=1}^n (X_i - m) \Bigr)_{n \in \N} $ are exponentially equivalent at speed $b_n^2$.
\end{proof}

\subsection{Proofs of Theorems~\ref{ThmClt1},~\ref{ThmBE1} and  \ref{ThmMDPRkn}}
\label{SectionProofs}

\emph{General observation.} If \(\X \sim \GenDist\), then, because of Proposition~\ref{PropGenDistFJP}, we can write \(\X \GlVert \frac{R^{(n)} Y_p^{(n)}}{\lVert Y_p^{(n)} \rVert_p}\), and then
\begin{equation*}
U_n(\phi^{(k)})(\X) \GlVert \binom{n}{k}^{-1} \sum_{i \in J_k^{(n)}} \prod_{j = 1}^k \frac{(R^{(n)})^p \lvert Y_{j, p} \rvert^p}{\lVert Y_p^{(n)} \rVert_p^p} = \frac{(R^{(n)})^{k p}}{\lVert Y_p^{(n)} \rVert_p^{k p}} U_n(\phi^{(k)})(Y_p^{(n)});
\end{equation*}
this leads to
\begin{equation*}
\Ratio{k_1, k_2, p}{(n)}(\X) = \frac{U_n(\phi^{(k_2)})(\X)^{1/(k_2 p)}}{U_n(\phi^{(k_1)})(\X)^{1/(k_1 p)}} \GlVert \frac{U_n(\phi^{(k_2)})(Y_p^{(n)})^{1/(k_2 p)}}{U_n(\phi^{(k_1)})(Y_p^{(n)})^{1/(k_1 p)}} = \Ratio{k_1, k_2, p}{(n)}(Y_p^{(n)}).
\end{equation*}
In essence, this means that instead of working with \(\X \sim \GenDist\) it suffices to consider \((X_i)_{i \in \N}\) i.i.d.\ \(p\)\-/Gaussian. Of course this has no impact on the case \(\X \sim \Surf_p^{(n)}\); indeed things are more delicate as we do not have a probabilistic representation. In order to establish moderate deviations results for the surface measure $\Surf_p^{(n)}$, we will need the following exponential equivalence of the cone measure $\Keg_p^{(n)}$ and the surface measure $\Surf_p^{(n)}$.

\begin{lem}
    \label{PropConeSurfExpEqui}
Let $p \in [1, \infty)$, for each $n \in \N$ let $D_n \in \Borel(\Sph{p}{n-1})$, and let $(s_n)_{n \in \N}$ be a sequence of positive real numbers such that
\begin{equation*}
    \lim_{n \rightarrow \infty} \frac{\log(n)}{s_n} = 0.
\end{equation*}
Then, we have
\begin{align*}
    \limsup_{n \rightarrow \infty} \frac{1}{s_n } \log ( \Surf_p^{(n)}(D_n)) 
    & = \limsup_{n \rightarrow \infty} \frac{1}{s_n } \log ( \Keg_p^{(n)}(D_n)),\\
    \liminf_{n \rightarrow \infty} \frac{1}{s_n } \log ( \Surf_p^{(n)}(D_n)) 
    & = \liminf_{n \rightarrow \infty} \frac{1}{s_n } \log ( \Keg_p^{(n)}(D_n)), 
\end{align*}
where we use the convention $\log(0) := - \infty$.
\end{lem}

\begin{proof}
Let \(n \in \N\). By Proposition~\ref{PropSurfKegDichte} there exists a constant $c \in [0, \frac{1}{2}]$ such that $n^{-c} \leq \frac{\dd\!\Surf_p^{(n)}}{\dd\!\Keg_p^{(n)}}(x) \leq n^c$ for all $x \in \Sph{p}{n-1}$. Thus, for any $D_n \in \Borel(\Sph{p}{n-1})$, we have $ n^{-c} \Keg_p^{(n)}(D_n) \leq \Surf_p^{(n)}(D_n) \leq n^c \Keg_p^{(n)}(D_n)$, which leads to
\begin{equation*}
\frac{-c \log(n)}{s_n} + \frac{1}{s_n} \log(\Keg_p^{(n)}(D_n)) \leq \frac{1}{s_n} \log(\Surf_p^{(n)}(D_n)) \leq \frac{c \log(n)}{s_n} + \frac{1}{s_n} \log(\Keg_p^{(n)}(D_n)).
\end{equation*}
Since by assumption $\lim_{n \rightarrow \infty} \frac{\log(n)}{s_n} = 0$, we can take the limits superior and inferior as $n \rightarrow \infty$ to get the claimed result.  
\end{proof}

\begin{lem}\label{LemTaylorRatio}
Let $k \in \N$ be fixed, $p \in [1, \infty)$, and for each \(n \in \N\) let $\X \sim \GenDist$. Then for the corresponding Maclaurin ratio $ \Ratio{k,n,p}{(n)}$ we have, for any \(n \geq k\),
\begin{equation*}
\Ratio{k,n,p}{(n)} \GlVert \ee^{m_p} \left( 1+ \frac{V_n^{(0)}}{\sqrt{n}} - \frac{V_n^{(k)}}{pk \sqrt{n}} + R \left( \frac{V_n^{(0)}}{\sqrt{n}}, \frac{V_n^{(k)}}{\sqrt{n}} \right)\right),
\end{equation*}
where $m_p$ is defined in Equation~\eqref{DefRelQuant},
\begin{align*}
V_n^{(0)} &= \sqrt{n} \bigl( U_n(\phi^{(0)})(\lvert Y_1 \rvert, \dotsc, \lvert Y_n \rvert) - m_p \bigr),\\
V_n^{(k)} &= \sqrt{n} \bigl( U_n(\phi^{(k)})(\lvert Y_1 \rvert^p, \dotsc, \lvert Y_n \rvert^p) - 1 \bigr),
\end{align*}
with respective kernels $ \phi^{(0)} \colon \R_+ \rightarrow \R$, $x \mapsto \log(x)$, and $ \phi^{(k)} \colon \R_+^k \rightarrow \R$, $(y_1, \dotsc, y_k) \mapsto \prod_{i=1}^k y_i$; \((Y_i)_{i \in \N}\) is a sequence of i.i.d.\ \(p\)\=/Gaussian random variables, and \(R \colon \R^2 \to \R\) has the property that there exist constants $M, \delta \in \R_+$ such that, for any $(x,y) \in \R^2$ with $\lVert (x,y) \rVert_2 \leq \delta$, we have $\lvert R(x,y) \rvert \leq M \rVert (x,y) \rVert_2^2$. 
\end{lem} 

\begin{proof}
Let \(n \geq k\). By the general observation on page~\pageref{SectionProofs} we may work directly with \(\X = Y_p^{(n)}\). Notice that we have
\begin{equation*}
S_{n, p}^{(n)} = \Biggl( \prod_{i = 1}^n \lvert X_i \rvert \Biggr)^{1/n} = \exp\Biggl( \frac{1}{n} \sum_{i = 1}^n \log\lvert X_i \rvert \Biggr) = \exp\bigl( U_n(\phi^{(0)})(\lvert X_1 \rvert, \dotsc, \lvert X_n \rvert) \bigr)
\end{equation*}
and
\begin{equation*}
S_{k, p}^{(n)} = \Biggl( \binom{n}{k}^{-1} \sum_{i \in J_{n, k}} \prod_{j = 1}^k \lvert X_{i_j} \rvert^p \Biggr)^{1/(k p)} = \bigl( U_n(\phi^{(k)})(\lvert X_1 \rvert^p, \dotsc, \lvert X_n \rvert^p) \bigr)^{1/(k p)},
\end{equation*}
where the kernels \(\phi^{(0)}\) and \(\phi^{(k)}\) are adopted from the statement of the lemma. Adopting \(V_n^{(0)}\) and \(V_n^{(k)}\) as well, we can write 
\begin{equation}
\label{EqRepMclaurinRatio}
\Ratio{k,n,p}{(n)} \GlVert \frac{  \exp \left( m_p + \frac{V_n^{(0)}}{\sqrt{n}} \right)}{ \left (1 + \frac{V_n^{(k)}}{\sqrt{n}} \right)^{1/(kp)}} = \ee^{m_p} \, \frac{  \exp \left(\frac{V_n^{(0)}}{\sqrt{n}} \right)}{ \left (1 + \frac{V_n^{(k)}}{\sqrt{n}} \right)^{1/(kp)}}.
\end{equation}
Now consider the function $F \colon \R \times (-1, \infty)  \rightarrow \R $ defined by $(x,y) \mapsto \frac{\ee^x}{(1+ y)^{1/(kp)}}$. We note that $F$ is two times continuously differentiable on its domain and thus, by the multidimensional Taylor formula, we can expand $F$ around $0 \in \R^2$ in the following way,
\begin{equation*}
F(x,y) = 1 + x - \frac{y}{pk} + R(x,y),
\end{equation*}
where by Taylor's theorem there exist $M, \delta \in \R_+$ such that, for any $(x,y) \in \R^2$ with $\lVert (x,y)\rVert_2 \leq \delta$, we have $\lvert R(x,y)\rvert \leq M \lVert (x,y)\rVert_2^2$. The rest of the statement in Lemma~\ref{LemTaylorRatio} follows now by plugging in \eqref{EqRepMclaurinRatio} into the Taylor expansion of $F$.
\end{proof}

\begin{proof}[Proof of Theorem~\ref{ThmMDPRkn}]
First we assume $\X \sim \GenDist$. By Lemma~\ref{LemTaylorRatio}, we have, for all \(n \geq k\),
\begin{equation*}
\frac{\sqrt{n}}{b_n} \left( \ee^{-m_p} \Ratio{k,n,p}{(n)} -1 \right) \GlVert \frac{V_n^{(0)}}{b_n} - \frac{V_n^{(k)}}{pk {b_n}}  + \frac{\sqrt{n}}{b_n} \, R\left( \frac{V_n^{(0)}}{\sqrt{n}}, \frac{V_n^{(k)}}{\sqrt{n}} \right),
\end{equation*}
where \(m_p\), \(V_n^{(0)}\), \(V_n^{(k)}\), and \(R\) are as stated in Lemma~\ref{LemTaylorRatio}. By Lemma~\ref{PropExpEquivalence}, $ \left( \frac{\sqrt{n}}{b_n} R \left( \frac{V_n^{(0)}}{\sqrt{n}}, \frac{V_n^{(k)}}{\sqrt{n}} \right) \right)_{n \geq k}$ is exponentially equivalent to $0$ at speed $b_n^2$, and thus $\left( \frac{\sqrt{n}}{b_n} \left(\ee^{-m_p} \Ratio{k,n,p}{(n)} -1 \right) \right)_{n \geq k}$ and $\left( \frac{V_n^{(0)}}{b_n} - \frac{V_n^{(k)}}{pk {b_n}} \right)_{n \geq k}$ are exponentially equivalent at speed $b_n^2$; hence it suffices to prove that \(\Bigl( \frac{V_n^{(0)}}{b_n} - \frac{V_n^{(k)}}{k p b_n} \Bigr)_{n \geq k}\) satisfies the claimed MDP at speed \(b_n^2\) with GRF \(x \mapsto \frac{x^2}{2 s_p^2}\).

Since $ \Erw[\exp(t_0 \log\lvert Y_1\rvert)] = \Erw[\lvert Y_1 \rvert^{t_0}] < \infty$ for all $t_0 \in \R_+$, we know that $ \Bigl( \frac{\sqrt{n}}{b_n} \frac{V_n^{(0)}}{\sqrt{n}} \Bigr)_{n \geq k}$ satisfies an MDP at speed $b_n^2$ with GRF $x \mapsto \frac{x^2}{2 \Var[\log\lvert Y_1 \rvert]}$. Since \(\Erw[\exp ( t_0 \lvert Y_1\rvert^p )] < \infty\) for all \(t_0 \in (0, \frac{1}{p})\), Proposition~\ref{LemMDPUStat} yields that $\Bigl( \frac{\sqrt{n}}{b_n} \frac{V_n^{(k)}}{\sqrt{n}} \Bigr)_{n \in \N}$ satisfies an MDP at speed $b_n^2$ with GRF $x \mapsto \frac{x^2}{2 k^2 \Var[\lvert Y_1\rvert^p]}$ (note \(\Erw[\lvert Y_1 \rvert^p] = 1\)). By Proposition~\ref{LemMDPUStat}, $ \left( \frac{V_n^{(k)}}{pk {b_n}} \right)_{n \geq k}$ and $ \left( \frac{1}{p b_n \sqrt{n}} \sum_{i=1}^k (\lvert Y_i\rvert^p -1) \right)_{n \geq k}$ are exponentially equivalent at speed $b_n^2$, and hence also $ \left( \frac{V_n^{(0)}}{b_n} - \frac{V_n^{(k)}}{pk b_n} \right)_{n \geq k}$ and $\Bigl( \frac{V_n^{(0)}}{b_n} - \frac{1}{p b_n \sqrt{n}} \sum_{i = 1}^n (\lvert Y_i \rvert^p - 1) \Bigr)_{n \geq k} = \Bigl( \frac{1}{b_n \sqrt{n}} \sum_{i=1}^n \Bigl( \log\lvert Y_i\rvert - \frac{1}{p} \lvert Y_i\rvert^p - m_p + \frac{1}{p} \Bigr) \Bigr)_{n \geq k}$ are exponentially equivalent at speed $b_n^2$. The latter is a sum of centered i.i.d.\ random variables with locally finite moment generating function. Thus, $\Bigl( \frac{1}{b_n \sqrt{n}} \sum_{i=1}^n \Bigl( \log\lvert Y_i\rvert - \frac{1}{p} \lvert Y_i\rvert^p - m_p + \frac{1}{p} \Bigr) \Bigr)_{n \geq k}$ satisfies an MDP at speed $b_n^2$ with GRF $x \mapsto \frac{x^2}{2 \Var[\log\lvert Y_1\rvert - \frac{1}{p} \lvert Y_1\rvert^p]} = \frac{x^2}{2 s_p^2}$.

Now consider the case where $\X \sim \Surf_p^{(n)}$; additionally let $\tilde{X}_p^{(n)} \sim \Keg_p^{(n)}$. By the first part of this proof, we know that $ \left( \frac{\sqrt{n}}{b_n} \left(\ee^{-m_p} \Ratio{k,n,p}{(n)}(\tilde{X}_p^{(n)}) - 1 \right) \right)_{n \geq k}$ satisfies an MDP at speed $b_n^2$. We are going to prove the upper bound of the MDP for $\Bigl( \frac{\sqrt{n}}{b_n} \Bigl( \ee^{-m_p} \Ratio{k,n,p}{(n)}(\X) - 1 \Bigr) \Bigr)_{n \geq k}$; the lower bound can be done analogously. To that end, let $A \in \Borel(\R)$ be closed, and we have to show
\begin{equation*}
\limsup_{n \rightarrow \infty} \frac{1}{b_n^2} \log \left( \prb\left[ \frac{\sqrt{n}}{b_n} \left(\ee^{-m_p} \Ratio{k,n,p}{(n)}(\X) -1\right) \in A \right] \right) \leq -\inf_{x \in A} \frac{x^2}{2 s_p^2}.
\end{equation*}
Define the set
\begin{equation*}
D_n := \left \{ x \in \Sph{p}{n-1} \Colon \frac{\sqrt{n}}{b_n} \left(\ee^{-m_p} \Ratio{k, n, p}{(n)}(x) - 1 \right) \in A  \right \}.
\end{equation*}
Then we can write
\begin{equation*}
\prb\left[ \frac{\sqrt{n}}{b_n} \left(\ee^{-m_p} \Ratio{k,n,p}{(n)}(\X) -1\right) \in A \right] = \Surf_p^{(n)}(D_n) \quad \text{and} \quad \prb\left[ \frac{\sqrt{n}}{b_n} \left(\ee^{-m_p} \Ratio{k,n,p}{(n)}(\tilde{X}_p^{(n)}) - 1 \right) \in A \right] = \Keg_p^{(n)}(D_n).
\end{equation*}
From Lemma~\ref{PropConeSurfExpEqui}, we get
\begin{equation*}
     \limsup_{n \rightarrow \infty} \frac{1}{b_n^2} \log \left( \Surf_p^{(n)} ( D_n) \right)  = \limsup_{n \rightarrow \infty} \frac{1}{b_n^2} \log \left( \Keg_p^{(n)} ( D_n) \right) \leq - \inf_{x \in A} \frac{x^2}{2 s_p^2},
\end{equation*}
where we also have used $\lim_{n \rightarrow \infty} \frac{\log(n)}{b_n^2} = 0$.

\end{proof}

\begin{proof}[Proof of Theorem~\ref{ThmBE1}]
First assume $\X \sim \GenDist$, or equivalently \(\X \sim \Gaussp{p}^{\otimes n}\), by the general observation on page~\pageref{SectionProofs}. For the associated Maclaurin ratio $\Ratio{k,n,p}{(n)}$, we obtain from Lemma~\ref{LemTaylorRatio}
\begin{equation*}
\Ratio{k,n,p}{(n)} \GlVert \ee^{m_p} \left( 1+ \frac{V_n^{(0)}}{\sqrt{n}} - \frac{V_n^{(k)}}{pk \sqrt{n}} + R \left( \frac{V_n^{(0)}}{\sqrt{n}}, \frac{V_n^{(k)}}{\sqrt{n}} \right)\right).
\end{equation*}
By rearranging the terms, we get
\begin{align*}
\sqrt{n} \left( \ee^{-m_p} \Ratio{k,n,p}{(n)} -1 \right) & = V_n^{(0)} - \frac{1}{pk} V_n^{(k)} + \sqrt{n} R\left( \frac{V_n^{(0)}}{\sqrt{n}}, \frac{V_n^{(k)}}{\sqrt{n}} \right)\\
& = \sqrt{n} \left( U_n( \phi^{(0)}) - m_p - \frac{1}{pk} U_n(\phi^{(k)}) +\frac{1}{pk} \right) + \sqrt{n} R\left( \frac{V_n^{(0)}}{\sqrt{n}}, \frac{V_n^{(k)}}{\sqrt{n}} \right)\\
& = \sqrt{n} \left( U_n( \tilde{\phi}^{(k)}) - m_p +\frac{1}{pk} \right) + \sqrt{n} R\left( \frac{V_n^{(0)}}{\sqrt{n}}, \frac{V_n^{(k)}}{\sqrt{n}} \right),
\end{align*}
where we have introduced the kernel $\tilde{\phi}^{(k)} \colon \R^k \rightarrow \R$, $(x_1, \dotsc, x_k) \mapsto \frac{1}{k} \sum_{i=1}^k \log\lvert x_i\rvert - \frac{1}{pk} \lvert x_1 \rvert^p \dotsm \lvert x_k \rvert^p$. Define \(\tilde{V}_n^{(k)} := \sqrt{n} \bigl( U_n(\tilde{\phi}^{(k)}) - m_p + \frac{1}{k p} \bigr)\), and let $Z \sim \Normal(0, s_p^2)$ and $\epsilon \in \R_+$, then we get by Proposition~\ref{PropEstimateKolmogoroff}
\begin{equation}
\label{EqdistKol}
    \dKol\left( \sqrt{n} \left( \ee^{-m_p} \Ratio{k,n,p}{(n)} -1  \right) , Z \right) \leq \dKol\left( \tilde{V}_n^{(k)}, Z \right) + \prb\left[ \sqrt{n} \left\lvert R\left( \frac{V_n^{(0)}}{\sqrt{n}}, \frac{V_n^{(k)}}{\sqrt{n}} \right) \right\rvert > \epsilon \right] + \frac{\epsilon }{\sqrt{2 \pi s_p^2 }}.
\end{equation}
$U_n( \tilde{\phi}^{(k)})$ is a U\=/statistic of order $k$ with $\Erw[\lvert U_n(\tilde{\phi}^{(k)}) \rvert^3] \leq 4 \Erw[\lvert\log \lvert X_1\rvert\rvert^3] + \frac{4}{p^3 k^3} \Erw[\lvert X_1\rvert^{3p}]^k < \infty$, since $X_1 \sim \Gaussp{p}$. Thus, by Proposition~\ref{PropUStat}, \ref{UStatBE}, $\tilde{V}_n^{(k)}$ satisfies a Berry\--Essen bound, i.e., there exists a constant $ C_1 \in \R_+$ such that $ \dKol\left( \tilde{V}_n^{(k)}, Z \right) \leq \frac{C_1}{\sqrt{n}}$ for all \(n \geq k\). The quantity $s_p^2$ from Equation~\eqref{DefRelQuant} is indeed the correct variance, since the sequence $(V_n^{(0)}, V_n^{(k)})_{n \geq k}$ satisfies the CLT for U\=/statistics (see Proposition~\ref{PropUStat}, \ref{USatCLT}), and therefore $\sqrt{n} \Bigl( U_n( \phi^{(0)}) - \frac{1}{pk} U_n( \phi^{(k)}) - m_p +\frac{1}{pk} \Bigr) \KiVert{} \Normal(0, s_p^2)$.

Now we give an upper bound of the term in \eqref{EqdistKol} containing the remainder $R$ from the Taylor expansion. Since there exist constants $M, \delta \in \R_+$ such that $ \lvert R(x,y)\rvert \leq M \lVert(x,y)\rVert_2^2$ provided $\lVert(x,y)\rVert_2^2 \leq \delta$, we can see that
\begin{equation*}
\prb\left[ \sqrt{n} \left\lvert R \left( \frac{V_n^{(0)}}{\sqrt{n}}, \frac{V_n^{(k)}}{\sqrt{n}} \right) \right\rvert > \epsilon \right] \leq \prb\left[ \left\lVert \left ( \frac{V_n^{(0)}}{\sqrt{n}}, \frac{V_n^{(k)}}{\sqrt{n}} \right) \right\rVert_2^2 > \frac{\epsilon}{M \sqrt{n}} \right] + \prb\left[ \left\lVert \left ( \frac{V_n^{(0)}}{\sqrt{n}}, \frac{V_n^{(k)}}{\sqrt{n}} \right) \right\rVert_2^2 > \delta \right].
\end{equation*}
We now choose \(\epsilon = \frac{b_n^2}{\sqrt{n}}\), with some sequence \((b_n)_{n \in \N}\) of positive reals such that \(\lim_{n \to \infty} b_n = \infty\) and \(\lim_{n \to \infty} \frac{b_n}{n^{1/4}} = 0\) to be determined later. Therewith, we get
\begin{align*}
\prb\left[ \sqrt{n} \left\lvert R \left( \frac{V_n^{(0)}}{\sqrt{n}}, \frac{V_n^{(k)}}{\sqrt{n}} \right) \right\rvert > \epsilon \right] &\leq \prb\left[ \left\lVert \left ( \frac{V_n^{(0)}}{\sqrt{n}}, \frac{V_n^{(k)}}{\sqrt{n}} \right) \right\rVert_2^2 > \frac{b_n^2}{M n} \right] + \prb\left[ \left\lVert \left ( \frac{V_n^{(0)}}{\sqrt{n}}, \frac{V_n^{(k)}}{\sqrt{n}} \right) \right\rVert_2^2 > \delta \right]\\
&\leq 2 \prb\left[ \bigl\lVert (V_n^{(0)}, V_n^{(k)}) \bigr\rVert_2 > \frac{b_n}{\sqrt{M}} \right]\\
&\leq 2 \prb\left[ \lvert V_n^{(0)} \rvert > \frac{b_n}{2 \sqrt{M}} \right] + 2 \prb\left[ \lvert V_n^{(k)} \rvert > \frac{b_n}{2 \sqrt{M}} \right],
\end{align*}
where we have used \(\frac{b_n^2}{\sqrt{n}} \leq \delta\) for all \(n \geq k\) large enough, and the inequality \(\lVert (x, y) \rVert_2 \leq \lVert (x, y) \rVert_1 = \lvert x \rvert + \lvert y \rvert\), valid for all \((x, y) \in \R^2\), plus the usual splitting argument. By Proposition~\ref{LemMDPUStat}, either sequence $\left( \frac{1}{b_n} V_n^{(0)} \right)_{n \geq k}$ and $ \left( \frac{1}{b_n} V_n^{(k)} \right)_{n \geq k}$ satisfies an MDP at speed $b_n^2$ with GRF $x \mapsto \frac{x^2}{2 \sigma_i^2}$ with some appropriate $\sigma_i^2 \in \R_+$, for $i \in \{0, k\}$. Therefore
\begin{equation*}
\limsup_{n \to \infty} \frac{1}{b_n^2} \log \prb\left[ \lvert V_n^{(0)} \rvert > \frac{b_n}{2 \sqrt{M}} \right] \leq -\inf_{\lvert x \rvert \geq 1/(2 \sqrt{M})} \frac{x^2}{2 \sigma_0^2} = -\frac{1}{8 M \sigma_0^2},
\end{equation*}
and thus there exists a \(n_0 \geq k\) such that, for any \(n \geq n_0\),
\begin{equation*}
\prb\left[ \lvert V_n^{(0)} \rvert > \frac{b_n}{2 \sqrt{M}} \right] \leq \exp\left( -\frac{b_n^2}{16 M \sigma_0^2} \right);
\end{equation*}
the analogous estimate holds true for \(V_n^{(k)}\), perhaps with a different \(n_k \geq k\). So take \(n \geq \max\{n_0, n_k\}\) and choose \(b_n^2 := 8 M \log(n) \max\{\sigma_0^2, \sigma_k^2\}\), then we obtain, for all \(n \geq k\),
\begin{equation*}
\prb\left[ \sqrt{n} \left\lvert R \left( \frac{V_n^{(0)}}{\sqrt{n}}, \frac{V_n^{(k)}}{\sqrt{n}} \right) \right\rvert > \epsilon \right] \leq 2 \prb\left[ \lvert V_n^{(0)} \rvert > \frac{b_n}{2 \sqrt{M}} \right] + 2 \prb\left[ \lvert V_n^{(k)} \rvert > \frac{b_n}{2 \sqrt{M}} \right] \leq \frac{C_2}{\sqrt{n}},
\end{equation*}
where \(C_2 \in \R_+\) is an appropriate constant. Referring back to Equation~\eqref{EqdistKol} and recalling \(\epsilon = \frac{b_n^2}{\sqrt{n}}\), now we see that there exists a \(C_3 \in \R_+\) such that, for all \(n \geq k\),
\begin{equation*}
\dKol\left( \tilde{V}_n^{(k)}, Z \right) + \prb\left[ \sqrt{n} \left\lvert R\left( \frac{V_n^{(0)}}{\sqrt{n}}  , \frac{V_n^{(k)}}{\sqrt{n}} \right) \right\rvert > \epsilon \right] + \frac{\epsilon }{\sqrt{2 \pi s_p^2 }} \leq \frac{C_3 \sqrt{\log(n)}}{\sqrt{n}}.
\end{equation*}

Now we consider the case $\X \sim \Surf_p^{(n)}$. For each \(n \geq k\) take \(\tilde{X}_p^{(n)} \sim \Keg_p^{(n)}\), then by the first part of this proof, there exists a constant $c \in \R_+$ such that $\dKol\bigl( \sqrt{n} ( \ee^{-m_p} \Ratio{k,n,p}{(n)}(\tilde{X}_p^{(n)}) - 1  ), Z \bigr) \leq \frac{c \sqrt{\log(n)}}{\sqrt{n}}$ for all $n \geq k$, where $Z \sim \Normal(0, s_p^2)$. For any \(t \in \R\) introduce the set $A_{t, n} := \bigl\{ x \in \Sph{p}{n-1} \Colon \sqrt{n} \bigl( \ee^{-m_p} \Ratio{k,n,p}{(n)}(x) - 1 \bigr) \geq t \bigr\}$, then we are lead to
\begin{align*}
\dKol\left( \sqrt{n} \left ( \ee^{-m_p} \Ratio{k,n,p}{(n)}(\X) - 1  \right ), Z \right) & \leq \dKol\left( \sqrt{n} \left ( \ee^{-m_p} \Ratio{k,n,p}{(n)}(\X) - 1  \right ), \sqrt{n} \left (\ee^{-m_p} \Ratio{k,n,p}{(n)}(\tilde{X}_p^{(n)}) - 1  \right ) \right)\\
&\quad + \dKol\left( \sqrt{n} \left (\ee^{-m_p} \Ratio{k,n,p}{(n)}(\tilde{X}_p^{(n)}) - 1 \right ), Z \right) \\
&\leq \sup_{t \in \R} \left\lvert \Surf_p^{(n)}(A_{t, n}) - \Keg_p^{(n)}(A_{t, n}) \right\rvert + \frac{c \sqrt{\log(n)}}{\sqrt{n}} \\
& \leq \frac{c_p}{\sqrt{n}} + \frac{c \sqrt{\log(n)}}{\sqrt{n}}\\
&\leq \frac{C \sqrt{\log(n)}}{\sqrt{n}},
\end{align*}
where we have used Proposition~\ref{PropSurfMeasConeMeasClose} to bound $\sup_{t \in \R} \bigl\lvert \Surf_p^{(n)}(A_{t, n}) - \Keg_p^{(n)}(A_{t, n}) \bigr\rvert$ from above by $\frac{c_p}{\sqrt{n}}$, and lastly have chosen $C \in \R_+$ such that the last inequality holds true for all $n \geq k$.
\end{proof}

\begin{proof}[Proof of Theorem~\ref{ThmClt1}]
This now is an easy consequence of Theorem~\ref{ThmBE1}, since convergence w.r.t.\ the Kolmogorov distance implies convergence in distribution.
\end{proof}

We also provide an alternate, direct proof here, where we assume slightly different, partly more general premises. It is going to use the following auxiliary statement, which in general appears to be well\-/known and widely used; nevertheless for the convenience of the reader we provide a version with its proof that suits precisely our needs.

\begin{lem}\label{LemRestgliedNull}
Let \(d, l \in \N\), let \(R \colon \R^d \to \R\) be a function for which there exist \(M, \delta \in \R_+\) such that \(\lvert R(x) \rvert \leq M \lVert x \rVert_2^l\) for all \(x \in \R^d\) with \(\lVert x \rVert_2 \leq \delta\), and let \((Z_n)_{n \in \N}\) be a sequence of \(\R^d\)\=/valued random variables such that \(\lim_{n \to \infty} n^{1/2 - 1/(2 l)} Z_n = 0\) in probability. Then \(\lim_{n \to \infty} n^{(l - 1)/2} R(Z_n) = 0\) in probability.
\end{lem}

\begin{proof}
Let \(\epsilon \in \R_+\), then for all \(n \geq 1\),
\begin{align*}
\prb\bigl[ n^{(l - 1)/2} \lvert R(Z_n) \rvert \geq \epsilon \bigr] &= \prb\bigl[ n^{(l - 1)/2} \lvert R(Z_n) \rvert \geq \epsilon \wedge \lVert Z_n \rVert_2 \leq \delta \bigr] + \prb\bigl[ n^{(l - 1)/2} \lvert R(Z_n) \rvert \geq \epsilon \wedge \lVert Z_n \rVert_2 > \delta \bigr]\\
&\leq \prb\bigl[ n^{(l - 1)/2} M \lVert Z_n \rVert_2^l \geq \epsilon \bigr] + \prb\bigl[ \lVert Z_n \rVert_2 \geq \delta \bigr]\\
&= \prb\Biggl[ \bigl\lVert n^{1/2 - 1/(2 l)} Z_n \bigr\rVert_2 \geq \biggl( \frac{\epsilon}{M} \biggr)^{1/l} \Biggr] + \prb\bigl[ \bigl\lVert n^{1/2 - 1/(2 l)} Z_n \bigr\rVert_2 \geq n^{1/2 - 1/(2 l)} \delta \bigr].
\end{align*}
Because of \(\lim_{n \to \infty} n^{1/2 - 1/(2 l)} Z_n = 0\) in probability, we get immediately
\begin{equation*}
\lim_{n \to \infty} \prb\Biggl[ \bigl\lVert n^{1/2 - 1/(2 l)} Z_n \bigr\rVert_2 \geq \biggl( \frac{\epsilon}{M} \biggr)^{1/l} \Biggr] = 0.
\end{equation*}
Also we have \(n^{1/2 - 1/(2 l)} \geq 1\) for all \(n \in \N\), and this implies
\begin{equation*}
\limsup_{n \to \infty} \prb\bigl[ \bigl\lVert n^{1/2 - 1/(2 l)} Z_n \bigr\rVert_2 \geq n^{1/2 - 1/(2 l)} \delta \bigr] \leq \limsup_{n \to \infty} \prb\bigl[ \bigl\lVert n^{1/2 - 1/(2 l)} Z_n \bigr\rVert_2 \geq \delta \bigr] = 0,
\end{equation*}
and the claim follows.
\end{proof}

\begin{proof}[Alternate proof of Theorem~\ref{ThmClt1}]
Let \((X_i)_{i \in \N}\) be a sequence of i.i.d.\ real\-/valued random variables such that \(\Erw[(\log\lvert X_1 \rvert)^2] < \infty\) and \(\Erw[\lvert X_1 \rvert^{2 p}] < \infty\); write \(\mu := \Erw[\lvert X_1 \rvert^p]\). Then Lemma~\ref{LemTaylorRatio} still holds true up to minor modifications, namely we do not need a separate sequence \((Y_i)_{i \in \N}\) anymore as its role is fulfilled by \((X_i)_{n \in \N}\) itself; \(\phi^{(0)}\), \(\phi^{(k)}\), \(V_n^{(0)}\) remain as stated; and redefine \(V_n^{(k)} := \sqrt{n} \bigl( U_n(\phi^{(k)}) - \mu^k \bigr)\), because now \(\Erw[U_n(\phi^{(k)})] = \mu^k\). Then analogously to the original situation we have
\begin{align*}
\Ratio{k, n, p}{(n)} &= \frac{\exp\Bigl( m_p + \frac{V_n^{(0)}}{\sqrt{n}} \Bigr)}{\Bigl( \mu^k + \frac{V_n^{(k)}}{\sqrt{n}} \Bigr)^{1/(k p)}} = \frac{\ee^{m_p}}{\mu^{1/p}} \, \frac{\exp\Bigl( \frac{V_n^{(0)}}{\sqrt{n}} \Bigr)}{\Bigl( 1 + \frac{V_n^{(k)}}{\mu^k \sqrt{n}} \Bigr)^{1/(k p)}}\\
&= \frac{\ee^{m_p}}{\mu^{1/p}} \Biggl( 1 + \frac{V_n^{(0)}}{\sqrt{n}} - \frac{V_n^{(k)}}{k p \mu^k \sqrt{n}} + R\Biggl( \frac{V_n^{(0)}}{\sqrt{n}}, \frac{V_n^{(k)}}{\mu^k \sqrt{n}} \Biggr) \Biggr),
\end{align*}
where \(R\) is the same as in Lemma~\ref{LemTaylorRatio} (including \(\delta\) and \(M\)), and therefore
\begin{equation*}
\sqrt{n} \Biggl( \frac{\mu^{1/p}}{\ee^{m_p}} \, \Ratio{k, n, p}{(n)} - 1 \Biggr) = V_n^{(0)} - \frac{V_n^{(k)}}{k p \mu^k} + \sqrt{n} R\Biggl( \frac{V_n^{(0)}}{\sqrt{n}}, \frac{V_n^{(k)}}{\mu^k \sqrt{n}} \Biggr).
\end{equation*}
By the CLT for U\=/statistics (Proposition~\ref{PropUStat}, \ref{USatCLT}), we know
\begin{equation*}
\begin{pmatrix} V_n^{(0)} \\ V_n^{(k)} \end{pmatrix} \KiVert{n \to \infty} \Normal(0, \Sigma)
\end{equation*}
with the covariance matrix
\begin{equation*}
\Sigma = \begin{pmatrix} \Var[\log\lvert X_1 \rvert] & k \mu^{k - 1} \cov[\log\lvert X_1 \rvert, \lvert X_1 \rvert^p] \\ k \mu^{k - 1} \cov[\log\lvert X_1 \rvert, \lvert X_1 \rvert^p] & k^2 \mu^{2 k - 2} \Var[\lvert X_1 \rvert^p] \end{pmatrix}.
\end{equation*}
From this we infer
\begin{equation*}
V_n^{(0)} - \frac{V_n^{(k)}}{k p \mu^k} \KiVert{n \to \infty} \Normal(0, s_p^2),
\end{equation*}
where \(s_p^2\) has been adjusted to
\begin{equation*}
s_p^2 = \Var\Biggl[ \log\lvert X_1 \rvert - \frac{\lvert X_1 \rvert^p}{p \mu} \Biggr].
\end{equation*}
Concerning the remainder term, recall that it satisfies \(\lvert R(x, y) \rvert \leq M \lVert (x, y) \rVert_2^2\) for all \((x, y) \in \R^2\) with \(\lVert (x, y) \rVert_2 \leq \sqrt{\delta}\); moreover \((V_n^{(0)}, V_n^{(k)})\) converges in distribution as \(n \to \infty\), and hence \(n^{1/4} \bigl( \frac{V_n^{(0)}}{\sqrt{n}}, \frac{V_n^{(k)}}{\mu^k \sqrt{n}} \bigr) = n^{-1/4} \bigl( V_n^{(0)}, \frac{V_n^{(k)}}{\mu^k} \bigr)\) converges to zero in distribution as \(n \to \infty\) per Slutsky's theorem, and hence also in probability. Thus, via Lemma~\ref{LemRestgliedNull} with \(l = 2\), we infer that \(\sqrt{n} R\bigl( \frac{V_n^{(0)}}{\sqrt{n}}, \frac{V_n^{(k)}}{\mu^k \sqrt{n}} \bigr)\) converges to zero in probability as \(n \to \infty\), and the proof is complete.
\end{proof}

\subsection{Proof of Theorem~\ref{ThmClt2}}
\label{SectionProofs2}

Recall that following the general observation at the beginning of Section~\ref{SectionProofs}, we may consider \(\Ratio{k_1, k_2, p}{(n)}(X_1, \dotsc, X_n)\) with \((X_i)_{i \in \N}\) a sequence of i.i.d.\ \(p\)\=/Gaussian random variables.

We start with Hoeffding's decomposition (Proposition~\ref{PropUStat}, \ref{UStatHdecomp}),
\begin{equation*}
(S_{k_1, p}^{(n)})^{k_1 p} = U_n(\phi^{(k_1)}) = 1 + \sum_{h = 1}^{k_1} \binom{k_1}{h} U_n(\phi_h^{(k_1)}),
\end{equation*}
with the kernel \(\phi^{(k_1)} \colon [0, \infty)^{k_1} \to \R\), \((x_i)_{i = 1}^{k_1} \mapsto \prod_{i = 1}^{k_1} \lvert x_i \rvert^p\), and analogously for \(k_2\), so also in the sequel. Via induction on the parameter \(h\), one shows
\begin{equation*}
\phi_h^{(k_1)}(x_1, \dotsc, x_h) = \prod_{i = 1}^h (\lvert x_i \rvert^p - 1);
\end{equation*}
in particular, this implies
\begin{equation*}
U_n(\phi_h^{(k_1)}) = U_n(\phi_h^{(k_2)}) = \binom{n}{h}^{-1} \sum_{i \in J_k^{(n)}} \prod_{j = 1}^h (\lvert X_{i_j} \rvert^p - 1).
\end{equation*}
We also know from Hoeffding's decomposition that \(\Var[U_n(\phi_h^{(k_1)})]\) is of order \(n^{-h}\). In what follows, we abbreviate \(Y_i := \lvert X_i \rvert^p - 1\) for \(i \in \N\); these random variables are independent and identically distributed, and \(\Erw[Y_1] = 0\) and \(\Erw[Y_1^2] = \Var[Y_1] = p\). For our purposes it suffices to take \(h \in \{1, 2, 3\}\) into account; in these cases the U\=/statistics can be written in terms of sums of powers, namely,
\begin{equation}\label{EqU1U2U3}
\begin{split}
U_n(\phi_1^{(k_1)}) &= \frac{1}{n} \sum_{i = 1}^n Y_i,\\
U_n(\phi_2^{(k_1)}) &= \frac{1}{n (n - 1)} \Biggl( \biggl( \sum_{i = 1}^n Y_i \biggr)^2 - \sum_{i = 1}^n Y_i^2 \Biggr),\\
U_n(\phi_3^{(k_1)}) &= \frac{1}{n (n - 1) (n - 2)} \Biggl( \biggl( \sum_{i = 1}^n Y_i \biggr)^3 - 3 \sum_{i = 1}^n Y_i \sum_{i = 1}^n Y_i^2 + 2 \sum_{i = 1}^n Y_i^3 \Biggr).
\end{split}
\end{equation}
The multivariate CLT yields
\begin{equation}\label{EqZGSZ1Z2}
\left( \begin{pmatrix} Z_n^{(1)} \\ Z_n^{(2)} \end{pmatrix} \right)_{n \in \N} := \left( \frac{1}{\sqrt{n}} \sum_{i = 1}^n \begin{pmatrix} Y_i \\ Y_i^2 - p \end{pmatrix} \right)_{n \in \N} \KiVert{} \begin{pmatrix} N_1 \\ N_2 \end{pmatrix} \sim \Normal\left( \begin{pmatrix} 0 \\ 0 \end{pmatrix}, \begin{pmatrix} p & c_{1, 2} \\ c_{1, 2} & v_2 \end{pmatrix} \right),
\end{equation}
where we additionally have introduced \(c_{1, 2} := \cov[Y_1, Y_1^2]\) and \(v_2 := \Var[Y_1^2]\). Plugging this into the representations~\eqref{EqU1U2U3} of the \(U_n(\phi_h^{(k_1)})\) leads to
\begin{equation}\label{EqV1V2V3}
\begin{split}
V_n^{(1)} &:= \sqrt{n} \, U_n(\phi_1^{(k_1)}) = Z_n^{(1)},\\
V_n^{(2)} &:= n U_n(\phi_2^{(k_1)}) = \frac{n}{n - 1} \Bigl( (Z_n^{(1)})^2 - p - \frac{Z_n^{(2)}}{\sqrt{n}} \Bigr),\\
V_n^{(3)} &:= n^{3/2} \, U_n(\phi_3^{(k_1)}) = \frac{n^2}{(n - 1) (n - 2)} \biggl( (Z_n^{(1)})^3 - 3 Z_n^{(1)} \Bigl( p + \frac{Z_n^{(2)}}{\sqrt{n}} \Bigr) + \frac{2}{n^{3/2}} \sum_{i = 1}^n Y_i^3 \biggr),\\
V_n^{(4, i)} &:= n^2 \sum_{h = 4}^{k_i} \binom{k_i}{h} U_n(\phi_h^{(k_i)}).
\end{split}
\end{equation}
With Slutsky's theorem, we get \((n^{-1/2} \, Z_n^{(2)})_{n \in \N} \to 0\) in probability, according to the strong law of large numbers \(\bigl( n^{-1} \sum_{i = 1}^n Y_i^3 \bigr)_{n \in \N} \to \Erw[Y_1^3]\) a.s.~and hence \(\bigl( n^{-3/2} \sum_{i = 1}^n Y_i^3 \bigr)_{n \geq 1} \to 0\) a.s., and therefrom and from~\eqref{EqZGSZ1Z2} we conclude
\begin{equation}\label{EqZGSV1V2V3}
\left( \begin{pmatrix} V_n^{(1)} \\ V_n^{(2)} \\ V_n^{(3)} \\ V_n^{(4, i)} \end{pmatrix} \right)_{n \geq 4} \KiVert{} \begin{pmatrix} N_1 \\ N_1^2 - p \\ N_1^3 - 3 p \, N_1 \\ H(N_1) \end{pmatrix},
\end{equation}
where the convergence of \((V_n^{(4, i)})_{n \geq k_2}\) is argued with Lee~\cite[Section~3.2.3]{leeUStatistics}; see there for further details, especially regarding the function \(H\). \(N_2\) is going to be needed later.

Returning to Hoeffding's decomposition we now have
\begin{equation*}
(S_{k_1, p}^{(n)})^{k_1 p} = 1 + k_1 \, \frac{V_n^{(1)}}{\sqrt{n}} + \binom{k_1}{2} \frac{V_n^{(2)}}{n} + \binom{k_1}{3} \frac{V_n^{(3)}}{n^{3/2}} + \frac{V_n^{(4, i)}}{n^2}.
\end{equation*}
Consequently, everything after the constant term converges to zero in probability, and we are going to apply Taylor series expansion, keeping explicit track only of terms \(\xi_n\) for which \(\lim_{n \to \infty} n^{3/2} \xi_n = 0\) in probability does \emph{not} hold true.

Recall that for any \(\alpha \in \R\), we have
\begin{equation}\label{EqBinomReihe}
(1 + x)^\alpha = 1 + \alpha x + \frac{\alpha (\alpha - 1)}{2} \, x^2 + \frac{\alpha (\alpha - 1) (\alpha - 2)}{6} \, x^3 + \rho(x),
\end{equation}
where there exist \(M, \delta \in \R_+\) such that the remainder term satisfies \(\lvert \rho(x) \rvert \leq M x^4\) for all \(x \in \R\) with \(\lvert x \rvert \leq \delta\). In our case, \(x_n := k_1 \, \frac{V_n^{(1)}}{\sqrt{n}} + \binom{k_1}{2} \frac{V_n^{(2)}}{n} + \binom{k_1}{3} \frac{V_n^{(3)}}{n^{3/2}} + \frac{V_n^{(4, i)}}{n^2}\), so \(\Var[x_n] \leq \frac{C}{n}\) for all \(n \geq k_2\) with some constant \(C\in(0,\infty)\). Therefore, \(\Var[n^{3/8} x_n] \leq C n^{-1/4}\) and so by Chebyshev's inequality \((n^{3/8} x_n)_{n \geq k_2}\) converges to zero in probability, and hence via Lemma~\ref{LemRestgliedNull} (with \(l = 4\)) \((n^{3/2} \rho(x_n))_{n \geq k_2}\) converges to zero in probability. Also, after plugging \(x_n\) into~\eqref{EqBinomReihe} and expanding the powers, we get terms of the shape \(\frac{(V_n^{(i)})^a (V_n^{(j)})^b}{n^c}\), and since \(n^{3/2} \frac{(V_n^{(i)})^a (V_n^{(j)})^b}{n^c} = n^{3/2 - c} (V_n^{(i)})^a (V_n^{(j)})^b\) converges to zero in probability if \(c \geq 2\), we collect the latter together with the remainder term in a single variable, that is,
\begin{align*}
(S_{k_1, p}^{(n)})^{\pm 1} &= \Biggl( 1 + k_1 \, \frac{V_n^{(1)}}{\sqrt{n}} + \binom{k_1}{2} \frac{V_n^{(2)}}{n} + \binom{k_1}{3} \frac{V_n^{(3)}}{n^{3/2}} + \frac{V_n^{(4, i)}}{n^2} \Biggr)^{\pm 1/(k_1 p)}\\
&= 1 \pm \frac{1}{k_1 p} \Biggl( k_1 \, \frac{V_n^{(1)}}{\sqrt{n}} + \binom{k_1}{2} \frac{V_n^{(2)}}{n} + \binom{k_1}{3} \frac{V_n^{(3)}}{n^{3/2}} \Biggr)\\
&\quad + \frac{1 \mp k_1 p}{2 k_1^2 p^2} \Biggl( \frac{k_1^2 (V_n^{(1)})^2}{n} + k_1^2 (k_1 - 1) \frac{V_n^{(1)} V_n^{(2)}}{n^{3/2}} \Biggr)\\
&\quad \pm \frac{(1 \mp k_1 p) (1 \mp 2 k_1 p)}{6 k_1^3 p^3} \, \frac{k_1^3 (V_n^{(1)})^3}{n^{3/2}} + R_1\\
&= 1 \pm \frac{V_n^{(1)}}{p \sqrt{n}} \pm \frac{(k_1 - 1) V_n^{(2)}}{2 p n} \pm \frac{(k_1 - 1) (k_1 - 2) V_n^{(3)}}{6 p n^{3/2}}\\
&\quad + \frac{(1 \mp k_1 p) (V_n^{(1)})^2}{2 p^2 n} + \frac{(1 \mp k_1 p) (k_1 - 1) V_n^{(1)} V_n^{(2)}}{2 p^2 n^{3/2}}\\
&\quad \pm \frac{(1 \mp k_1 p) (1 \mp 2 k_1 p) (V_n^{(1)})^3}{6 p^3 n^{3/2}} + R_1,
\end{align*}
where \((n^{3/2} R_1)_{n \geq k_2} \to 0\) in probability; we also have an analogous expansion w.r.t.\ \(k_2\).

With all this at our disposal, we arrive at
\begin{align*}
\Ratio{k_1, k_2, p}{(n)} &= \Biggl( 1 + \frac{V_n^{(1)}}{p \sqrt{n}} + \frac{(k_2 - 1) V_n^{(2)}}{2 p n} + \frac{(k_2 - 1) (k_2 - 2) V_n^{(3)}}{6 p n^{3/2}}\\
&\qquad + \frac{(1 - k_2 p) (V_n^{(1)})^2}{2 p^2 n} + \frac{(1 - k_2 p) (k_2 - 1) V_n^{(1)} V_n^{(2)}}{2 p^2 n^{3/2}}\\
&\qquad + \frac{(1 - k_2 p) (1 - 2 k_2 p) (V_n^{(1)})^3}{6 p^3 n^{3/2}} + R_2 \Biggr)\\
&\quad \cdot \Biggl( 1 - \frac{V_n^{(1)}}{p \sqrt{n}} - \frac{(k_1 - 1) V_n^{(2)}}{2 p n} - \frac{(k_1 - 1) (k_1 - 2) V_n^{(3)}}{6 p n^{3/2}}\\
&\qquad + \frac{(1 + k_1 p) (V_n^{(1)})^2}{2 p^2 n} + \frac{(1 + k_1 p) (k_1 - 1) V_n^{(1)} V_n^{(2)}}{2 p^2 n^{3/2}}\\
&\qquad - \frac{(1 + k_1 p) (1 + 2 k_1 p) (V_n^{(1)})^3}{6 p^3 n^{3/2}} + R_1 \Biggr)\\
&= 1 - \frac{(k_2 - k_1) (V_n^{(1)})^2}{2 p n} + \frac{(k_2 - k_1) V_n^{(2)}}{2 p n} + \frac{(k_2 - k_1) (k_1 + k_2) (V_n^{(1)})^3}{3 p n^{3/2}}\\
&\quad - \frac{(k_2 - k_1) (k_1 + k_2 - 1) V_n^{(1)} V_n^{(2)}}{2 p n^{3/2}} + \frac{(k_2 - k_1) (k_1 + k_2 - 3) V_n^{(3)}}{6 p n^{3/2}} + R_3,
\end{align*}
where \(R_3\) collects all terms \(\xi_n\) such that \(\lim_{n \to \infty} n^{3/2} \xi_n = 0\) in probability, thus also \(\lim_{n \to \infty} n^{3/2} R_3 = 0\) in probability. We rearrange further,
\begin{align*}
\frac{n p}{k_2 - k_1} (\Ratio{k_1, k_2, p}{(n)} - 1) &= \frac{V_n^{(2)} - (V_n^{(1)})^2}{2} + \frac{(k_1 + k_2) (V_n ^{(1)})^3}{3 \sqrt{n}}\\
&\quad - \frac{(k_1 + k_2 - 1) V_n^{(1)} V_n^{(2)}}{2 \sqrt{n}} + \frac{(k_1 + k_2 - 3) V_n^{(3)}}{6 \sqrt{n}} + R_4,
\end{align*}
where now \(\lim_{n \to \infty} n^{1/2} R_4 = 0\) in probability. At this point we see \(\bigl( V_n^{(2)} - (V_n^{(1)})^2 \bigr)_{n \geq k_2} \KiVert{} -p\) from~\eqref{EqZGSV1V2V3}, and that all the other terms, when multiplied with \(\sqrt{n}\), converge to some random variable in distribution. Therefore we have to ensure that \(\bigl( \sqrt{n} \bigl( V_n^{(2)} - (V_n^{(1)})^2 + p \bigr) \bigr)_{n \geq k_2}\) converges in distribution, and we have to determine its limit. Now by both~\eqref{EqV1V2V3} and~\eqref{EqZGSZ1Z2}
\begin{align*}
\sqrt{n} \Bigl( V_n^{(2)} - (V_n^{(1)})^2 + p \Bigr) &= \sqrt{n} \Bigl( \frac{n}{n - 1} \Bigl( (Z_n^{(1)})^2 - p - \frac{Z_n^{(2)}}{\sqrt{n}} \Bigr) - (Z_n^{(1)})^2 + p \Bigr)\\
&= \frac{\sqrt{n}}{n - 1} \Bigl( (Z_n^{(1)})^2 - p \Bigr) - \frac{n Z_n^{(2)}}{n - 1}\\
&\KiVert{n \to \infty} -N_2.
\end{align*}
Thus, together with~\eqref{EqZGSV1V2V3}, we arrive at
\begin{align*}
\sqrt{n} \Bigl( \frac{n p}{k_2 - k_1} \bigl( \Ratio{k_1, k_2, p}{(n)} - 1 \bigr) + \frac{p}{2} \Bigr) &= \frac{\sqrt{n}}{2} \Bigl( V_n^{(2)} - (V_n^{(1)})^2 + p \Bigr) + \frac{(k_1 + k_2) (V_n^{(1)})^3}{3}\\
&\quad - \frac{(k_1 + k_2 - 1) V_n^{(1)} V_n^{(2)}}{2} + \frac{(k_1 + k_2 - 3) V_n^{(3)}}{6} + \sqrt{n} R_4\\
&\KiVert{n \to \infty} -\frac{N_2}{2} + \frac{(k_1 + k_2) N_1^3}{3} - \frac{(k_1 + k_2 - 1) N_1 (N_1^2 - p)}{2}\\
&\quad\qquad + \frac{(k_1 + k_2 - 3) (N_1^3 - 3 p N_1)}{6}\\
&= p N_1 - \frac{N_2}{2}.
\end{align*}
A routine calculation confirms that the limit has the claimed distribution. \qed

\begin{rem}
A closer inspection of the proof reveals that the statement can be generalized to such sequences \((X_i)_{i \in \N}\) of i.i.d.\ random variables that \(\Erw[\lvert X_1 \rvert^{4 p}] < \infty\); then, calling \(m := \Erw[\lvert X_1 \rvert^p]\) and \(v := \Var[\lvert X_1 \rvert^p]\), one can show
\begin{equation*}
\sqrt{n} \Biggl( \frac{n p}{k_2 - k_1} \bigl( \Ratio{k_1, k_2, p}{(n)} - 1 \bigr) + \frac{v}{2 m^2} \Biggr) \KiVert{n \to \infty} \Normal(0, \rho^2),
\end{equation*}
where one has
\begin{equation*}
\rho^2 = \frac{1}{4} \Var\Biggl[ \Biggl( \frac{\lvert X_1 \rvert^p - m}{m} - \frac{v}{m^2} \Biggr)^2 \Biggr].
\end{equation*}
The reader is invited to check that \(\rho\) is zero iff \(\lvert X_1 \rvert\) almost surely takes at most two values \(x_1\), \(x_2\); in the case \(x_1 \neq x_2\) under the condition \(\prb[\lvert X_1 \rvert = x_1] = \frac{x_2^p}{x_1^p + x_2^p}\).
\end{rem}

\subsection{Proofs of the corollaries}

\begin{proof}[Proof of Corollary~\ref{cor:CLT1}.]
Adopting the notation of Theorem~\ref{ThmClt1} we have
\begin{equation*}
\sqrt{n} \Bigl( \ee^{-m_p} \Ratio{k, n, p}{(n)} - 1 \Bigr) \KiVert{n \to \infty} Z,
\end{equation*}
and because the normal CDF \(\Phi\) is continuous everywhere, this implies, for any \(t \in \R\),
\begin{equation*}
\lim_{n \to \infty} \prb\Bigl[ \sqrt{n} \Bigl( \ee^{-m_p} \Ratio{k, n, p}{(n)} - 1 \Bigr) \geq t \Bigr] = \prb[Z \geq t] = 1 - \Phi\Bigl( \frac{t}{s_p} \Bigr).
\end{equation*}
In particular for \(t = 0\) this becomes, after some rearrangement,
\begin{equation*}
\lim_{n \to \infty} \prb\bigl[ \Ratio{k, n, p}{(n)} \geq \ee^{m_p} \bigr] = 1 - \Phi(0) = \frac{1}{2}.
\end{equation*}
Writing out \(\Ratio{k, n, p}{(n)}\) (see Equation~\eqref{EqDefRatioEulerMclaurin}) yields the claimed formula.
\end{proof}

\begin{proof}[Proof of Corollary~\ref{cor:CLT2}.]
From Theorem~\ref{ThmClt2} we have
\begin{equation*}
\sqrt{n} \biggl( \frac{n p}{k_2 - k_1} \Bigl( \Ratio{k_1, k_2, p}{(n)} - 1 \Bigr) + \frac{p}{2} \biggr) \KiVert{n \to \infty} Z,
\end{equation*}
and because the CDF of \(Z\) is continuous everywhere we even have uniform convergence of the CDFs, or also of the tail probabilities, that is,
\begin{equation*}
\lim_{n \to \infty} \sup_{t \in \R} \Biggl\lvert \prb\biggl[ \sqrt{n} \biggl( \frac{n p}{k_2 - k_1} \Bigl( \Ratio{k_1, k_2, p}{(n)} - 1 \Bigr) + \frac{p}{2} \biggr) \geq t \biggr] - \Bigl( 1 - \Phi\Bigl( \frac{t}{\rho_p} \Bigr) \Bigr) \Biggr\rvert = 0.
\end{equation*}
This implies that, for any \(t \in \R \cup \{\pm\infty\}\) and any sequence \((t_n)_{n \in \N} \in \R^\N\) with \(\lim_{n \to \infty} t_n = t\) we have
\begin{equation*}
\lim_{n \to \infty} \prb\biggl[ \sqrt{n} \biggl( \frac{n p}{k_2 - k_1} \Bigl( \Ratio{k_1, k_2, p}{(n)} - 1 \Bigr) + \frac{p}{2} \biggr) \geq t_n \biggr] = 1 - \Phi\Bigl( \frac{t}{\rho_p} \Bigr),
\end{equation*}
where we interpret \(\Phi(-\infty) := 0\) and \(\Phi(\infty) := 1\). Now choose \(t_n := -n\), then we see
\begin{equation*}
\lim_{n \to \infty} \prb\biggl[ \sqrt{n} \biggl( \frac{n p}{k_2 - k_1} \Bigl( \Ratio{k_1, k_2, p}{(n)} - 1 \Bigr) + \frac{p}{2} \biggr) \geq -n \biggr] = 1 - \Phi\Bigl( -\frac{\infty}{\rho_p} \Bigr) = 1.
\end{equation*}
On the left\-/hand side notice
\begin{equation*}
\prb\biggl[ \sqrt{n} \biggl( \frac{n p}{k_2 - k_1} \Bigl( \Ratio{k_1, k_2, p}{(n)} - 1 \Bigr) + \frac{p}{2} \biggr) \geq -n \biggr] = \prb\biggl[ \Ratio{k_1, k_2, p}{(n)} \geq 1 - \frac{k_2 - k_1}{p \sqrt{n}} - \frac{k_2 - k_1}{2 n} \biggr];
\end{equation*}
now let \(c \in (0, 1)\), then \(1 - \frac{k_2 - k_1}{p \sqrt{n}} - \frac{k_2 - k_1}{2 n} \geq c\) for eventually all \(n\), and hence eventually
\begin{equation*}
\prb\biggl[ \sqrt{n} \biggl( \frac{n p}{k_2 - k_1} \Bigl( \Ratio{k_1, k_2, p}{(n)} - 1 \Bigr) + \frac{p}{2} \biggr) \geq -n \biggr] \leq \prb\bigl[ \Ratio{k_1, k_2, p}{(n)} \geq c \bigr] \leq 1.
\end{equation*}
Letting \(n \to \infty\) and writing out \(\Ratio{k_1, k_2, p}{(n)}\) concludes the proof.
\end{proof}

\section*{Acknowledgement}
LF was supported by the Austrian Science Fund (FWF), projects P-32405
and P-35322. JP is supported by the German Research Foundation (DFG) under project
516672205 and by the Austrian Science Fund (FWF) under project P-32405.

\bibliographystyle{plain}
\bibliography{Mclaurin_bib.bib}

\end{document}